\documentclass[11pt,reqno]{amsart}

\newtheorem{thm}{Theorem}[section]
\newtheorem*{thm*}{Theorem}
\newtheorem{lem}[thm]{Lemma}
\newtheorem*{lem*}{Lemma}

\newtheorem{cor}[thm]{Corollary}

\newtheorem{prop}[thm]{Proposition}

\theoremstyle{definition}

\newtheorem{assump}[thm]{Assumption}
\newtheorem{case}{Case}\renewcommand{\thecase}{}
\newtheorem*{case*}{Case}

\newtheorem*{defn*}{Definition}

\newtheorem*{exmp*}{Example}

\newtheorem{rmk}[thm]{Remark}
\newtheorem*{rmk*}{Remark}
\newtheorem{step}{Step}\renewcommand{\thestep}{}

\theoremstyle{remark}

\makeatletter
\def\alphenumi{
  \def\theenumi{\alph{enumi}}
  \def\p@enumi{\theenumi}
  \def\labelenumi{(\@alph\c@enumi)}}
\makeatother

\makeatletter
\def\thecase{\@arabic\c@case}
\makeatother

\makeatletter
\def\thestep{\@arabic\c@step}
\makeatother

\newcount\hh
\newcount\mm
\mm=\time
\hh=\time
\divide\hh by 60
\divide\mm by 60
\multiply\mm by 60
\mm=-\mm
\advance\mm by \time
\def\hhmm{\number\hh:\ifnum\mm<10{}0\fi\number\mm}

\let\oldmarginpar\marginpar
\renewcommand\marginpar[1]{\-\oldmarginpar[\raggedleft\footnotesize #1]%
{\raggedright\footnotesize #1}}

\renewcommand\emptyset{\varnothing}

\newcommand\NN{\mathbb{N}}

\newcommand\RR{\mathbb{R}}

\newcommand\eps{\varepsilon}

\newcommand\dist{\operatorname{dist}}

\newcommand\supp{\operatorname{supp}}

\numberwithin{equation}{section}

\usepackage{amssymb}
\usepackage{hyperref}
\usepackage{mathrsfs}
\usepackage{slashed}
\usepackage[usenames]{color}
\usepackage{url}
\usepackage{verbatim}

\hypersetup{pdftitle={$C^0$-estimates and smoothness of solutions to the parabolic equation defined by Kimura operators}}

\hypersetup{pdfauthor={Camelia A. Pop}}

\begin{document}

\title[$C^0$-estimates and smoothness of solutions]{$C^0$-estimates and smoothness of solutions to the parabolic equation defined by Kimura operators}

\author[C. Pop]{Camelia A. Pop}
\address[CP]{Department of Mathematics, University of Pennsylvania, 209 South 33rd Street, Philadelphia, PA 19104-6395}
\email{cpop@math.upenn.edu}

\date{\today{ }\hhmm}

\begin{abstract}
Kimura diffusions serve as a stochastic model for the evolution of gene frequencies in population genetics. Their infinitesimal generator is an elliptic differential operator whose second-order coefficients matrix degenerates on the boundary of the domain. In this article, we consider the inhomogeneous initial-value problem defined by generators of Kimura diffusions, and we establish $C^0$-estimates, which allows us to prove that solutions to the inhomogeneous initial-value problem are smooth up to the boundary of the domain where the operator degenerates, even when the initial data is only assumed to be continuous.
\end{abstract}

%



\maketitle

\tableofcontents

\section{Introduction}
\label{sec:Introduction}
The evolution of gene frequencies is one of the central themes of research in population genetics, and one of the natural ways to model the changes of gene frequencies in a population is through the use of Markov chains and their continuous limits. This line of research was initiated by R. Fisher (1922), J. Haldane (1932), S. Wright (1931), and later extended by M. Kimura (1957). The stochastic processes involved in these works are continuous limits of discrete Markov processes, which are solutions to stochastic differential equations whose infinitesimal generator is a degenerate-elliptic partial differential operator. A rigorous understanding of the regularity of solutions to parabolic equations defined by such operators plays a central role in the study of various probabilistic properties of the associated stochastic models.

A wide extension of the generator of continuous limits of the Wright-Fisher model \cite{Fisher_1922, Wright_1931, Haldane_1932,  Kimura_1957, Kimura_1964, Shimakura_1981, Ethier_Kurtz, KarlinTaylor2}
was introduced in the work of C. Epstein and R. Mazzeo \cite{Epstein_Mazzeo_2010, Epstein_Mazzeo_annmathstudies}, where the authors build a suitable Schauder theory to prove existence, uniqueness and optimal regularity of solutions to the inhomogeneous initial-value problem defined by generalized Kimura diffusion operators acting on functions defined on compact manifolds with corners. In our work, we extend the regularity results obtained in \cite{Epstein_Mazzeo_2010, Epstein_Mazzeo_annmathstudies} by proving a priori local Schauder estimates of solutions, in which we control the higher-order H\"older norm of solutions in terms of their supremum norm (Theorem \ref{thm:Time_interior_estimate}). This result allows us to prove in Theorem \ref{thm:ExistenceUniqueness_continuous} that the solutions are smooth up to the portion of the boundary where the operator degenerates, even when the initial data is only assumed to be continuous, as opposed to H\"older continuous in \cite{Epstein_Mazzeo_2010, Epstein_Mazzeo_annmathstudies}. In the sequel, we describe our main results and their applications in more detail.

Let $\RR_+:=(0,\infty)$ and $S_{n,m}:=\RR_+^n\times\RR^m$, where $n$ and $m$ are nonnegative integers such that $n+m\geq 1$. While generalized Kimura diffusion operators act on functions defined on compact manifolds with corners \cite[\S 2]{Epstein_Mazzeo_annmathstudies}, from an analytical point of view and due to the fact that we are interested in the local properties of solutions, in our article, we consider a second-order elliptic differential operator of the form
\begin{equation}
\label{eq:Generator}
\begin{aligned}
Lu&= \sum_{i=1}^n \left(x_ia_{ii}(z)u_{x_ix_i} + b_i(z)u_{x_i}\right) + \sum_{i,j=1}^n x_ix_j \tilde a_{ij}(z)u_{x_ix_j} \\
&\quad+\sum_{i=1}^n\sum_{l=1}^m x_i c_{il}(z)u_{x_iy_l} +\sum_{k,l=1}^m d_{kl}(z)u_{y_ky_l}+ \sum_{l=1}^m e_l(z)u_{y_l},
\end{aligned}
\end{equation}
defined for all $z=(x,y)\in S_{n,m}$ and $u \in C^2(S_{n,m})$. Even though the operator $L$ is defined on $S_{n,m}$, we still call $L$ a generalized Kimura diffusion operator since it preserves the local properties of the Kimura diffusion operators arising in population genetics. The operator $L$ is not strictly elliptic as we approach the boundary of the domain $S_{n,m}$, because the smallest eigenvalue of the second-order coefficient matrix tends to $0$ proportional to the distance to the boundary of the domain. For this reason, the sign of the coefficient functions $b_i(z)$ along $\partial S_{n,m}$ plays a crucial role in the regularity of solutions, and we always assume that the drift coefficients $b_i(z)$ are nonnegative functions along $\partial S_{n,m}$. The precise technical conditions satisfied by the coefficients of the operator $L$ are described in Assumption \ref{assump:Coeff}.

We prove local a priori Schauder estimates of solutions to the inhomogeneous initial-value problem,
\begin{equation}
\label{eq:Inhom_initial_value_problem}
\begin{aligned}
u_t-Lu&=g\quad\hbox{on } (0,\infty)\times S_{n,m},\\
u(0,\cdot)&=f\quad\hbox{on } S_{n,m},
\end{aligned}
\end{equation}
which we then use to prove the smoothness of solutions on $(0,\infty)\times\bar S_{n,m}$, when the initial data, $f$, is assumed to be only \emph{continuous} on $\bar S_{n,m}$. The inhomogeneous initial-value problem \eqref{eq:Inhom_initial_value_problem} on compact manifolds with corners was studied by C. Epstein and R. Mazzeo in \cite{Epstein_Mazzeo_2010, Epstein_Mazzeo_annmathstudies}, where they build suitable anisotropic H\"older spaces \cite[Chapter 5]{Epstein_Mazzeo_annmathstudies} to account for the degeneracy of the operator, and establish existence, uniqueness and optimal regularity of solutions \cite[Chapters 3 and 10]{Epstein_Mazzeo_annmathstudies}, under the assumptions that the initial data, $f$, and the source function, $g$, belong to suitable H\"older spaces, and the coefficients of the differential operator $L$ are smooth and bounded functions. In our work, we prove that the solutions of the inhomogeneous initial-value problem \eqref{eq:Inhom_initial_value_problem} are smooth functions on $(0, \infty)\times\bar S_{n,m}$ and continuous on $[0,\infty)\times\bar S_{n,m}$, when the coefficients of the operator $L$ and the source function, $g$, are assumed smooth, but the initial data is only assumed to be \emph{continuous}, as opposed to H\"older continuous in \cite[\S 11.2]{Epstein_Mazzeo_annmathstudies}. In addition, we relax the assumption in \cite{Epstein_Mazzeo_annmathstudies} that the coefficients of the operator $L$ are smooth, and we only require that they are H\"older continuous. Under the new hypotheses, we derive higher-order a priori local Schauder estimates in Theorems \ref{thm:Time_interior_estimate} and \ref{thm:Interior_estimate}, and we prove existence of solutions in H\"older spaces in Theorem \ref{thm:ExistenceUniqueness}. The technical definition of the anisotropic H\"older spaces adapted to our framework is given in \S \ref{subsec:Holder_spaces}.

Our main results are

\begin{thm}[Local a priori Schauder estimates I]
\label{thm:Time_interior_estimate}
Let $\alpha\in(0,1)$ and $k\in\NN$. Then there is a positive constant, $r_0=r_0(\alpha,k,m,n)$, such that the following hold.
Let $r\in (0,r_0)$ and $0<T_0<T$. Suppose that the coefficients of the differential operator $L$ defined in \eqref{eq:Generator} satisfy Assumption \ref{assump:Coeff}. Then, there is a positive constant, $C=C(\alpha,\delta,k,K,m,n,r,T_0,T)$, such that for all $z^0 \in \bar S_{n,m}$, and all functions, $u\in C^{k,2+\alpha}_{WF}([T_0/2,T]\times\bar B_{2r}(z^0))$,  we have that
\begin{equation}
\label{eq:Time_interior_estimate}
\|u\|_{C^{k,2+\alpha}_{WF}([T_0,T]\times \bar B_r(z^0))} \leq 
C\left(\|u_t-Lu\|_{C^{k,\alpha}_{WF}([T_0/2,T]\times\bar B_{2r}(z^0))}  + \|u\|_{C([T_0/2,T]\times\bar B_{2r}(z^0))} \right).
\end{equation}
\end{thm}

\begin{thm}[Local a priori Schauder estimates II]
\label{thm:Interior_estimate}
Let $\alpha\in(0,1)$ and $k\in\NN$. Then there is a positive constant, $r_0=r_0(\alpha,k,m,n)$, such that the following hold. Let $r\in(0,r_0)$ and $T>0$. Suppose that the coefficients of the differential operator $L$ defined in \eqref{eq:Generator} satisfy Assumption \ref{assump:Coeff}. Then, there is a positive constant, $C=C(\alpha,\delta,k,K,m,n,r,T)$, such that for all $z^0 \in \bar S_{n,m}$, and all functions, $u\in C^{k,2+\alpha}_{WF}([0,T]\times\bar B_{2r}(z^0))$,  we have that
\begin{equation}
\label{eq:Interior_estimate}
\begin{aligned}
\|u\|_{C^{k,2+\alpha}_{WF}([0,T]\times \bar B_r(z^0))} 
&\leq C\left(\|u_t-Lu\|_{C^{k,\alpha}_{WF}([0,T]\times\bar B_{2r}(z^0))}\right.\\
&\quad\left.+ \|u(0,\cdot)\|_{C^{k,2+\alpha}_{WF}(\bar B_{2r}(z^0))} + \|u\|_{C([0,T]\times\bar B_{2r}(z^0))} \right).
\end{aligned}
\end{equation}
\end{thm}

\begin{rmk}[Comparison between Theorems \ref{thm:Time_interior_estimate} and \ref{thm:Interior_estimate}]
\label{rmk:Comparison_a_priori_estimates}
Notice that the a priori Schauder estimate \eqref{eq:Time_interior_estimate} shows that the $C^{k,2+\alpha}_{WF}([T_0,T]\times \bar B_r(z^0))$-H\"older norm of the function $u$ can be controlled in terms of its supremum norm on $[0,T]\times \bar B_{2r}(z^0)$, while estimate \eqref{eq:Interior_estimate} also involves the $C^{k,2+\alpha}_{WF}(\bar B_{2r}(z^0))$-H\"older norm of the initial condition, $u(0,\cdot)$. Thus, estimate \eqref{eq:Time_interior_estimate} implies that to prove higher-order H\"older regularity of solutions on $(0,T]\times\bar S_{n,m}$, it is sufficient to establish a control on the supremum norm of the solution $u$ on $[0,T]\times\bar S_{n,m}$, a key fact that we use in our proof of the smoothness of solutions to the initial-value problem \eqref{eq:Inhom_initial_value_problem} with continuous initial data, in Theorem \ref{thm:ExistenceUniqueness_continuous}.
\end{rmk}

We now state our results on existence and uniqueness of solutions with H\"older continuous initial data, and with only continuous initial data.

\begin{thm}[Existence and uniqueness of solutions with H\"older continuous initial data]
\label{thm:ExistenceUniqueness}
Let $\alpha\in(0,1)$, $k\in\NN$ and $T>0$. Suppose that the coefficients of the differential operator $L$ satisfy Assumption \ref{assump:Coeff}. Then, there is a positive constant, $C=C(\alpha,\delta,k,K,m,n,T)$, such that the following hold. Let $g \in C^{k,\alpha}_{WF}([0,T]\times\bar S_{n,m})$ and $f \in C^{k,2+\alpha}_{WF}(\bar S_{n,m})$. Then there is a unique solution, $u \in C^{k,2+\alpha}_{WF}([0,T]\times\bar S_{n,m})$, to the inhomogeneous initial-value problem \eqref{eq:Inhom_initial_value_problem}, and the function $u$ satisfies the Schauder estimate,
\begin{equation}
\label{eq:Solution_estimate}
\begin{aligned}
\|u\|_{C^{k,2+\alpha}_{WF}([0,T]\times \bar S_{n,m})} 
&\leq C\left(\|g\|_{C^{k,\alpha}_{WF}([0,T]\times\bar S_{n,m})}
 + \|f\|_{C^{k,2+\alpha}_{WF}(\bar S_{n,m})}  \right).
\end{aligned}
\end{equation}
\end{thm}

\begin{thm}[Existence and uniqueness of solutions with continuous initial data]
\label{thm:ExistenceUniqueness_continuous}
Let $T>0$. Suppose that the coefficients of the differential operator $L$ satisfy Assumption \ref{assump:Coeff}, for all $k\in\NN$. Let $g \in C^{\infty}([0,T]\times\bar S_{n,m})$ and $f \in C(\bar S_{n,m})$. Then there is a unique solution, $u \in C([0,T]\times\bar S_{n,m})\cap C^{\infty}((0,T]\times\bar S_{n,m})$, to the inhomogeneous initial-value problem \eqref{eq:Inhom_initial_value_problem}. Moreover, for all $\alpha\in (0,1)$, $k\in\NN$ and $T_0\in (0,T)$, there is a positive constant, $C=C(\alpha,\delta,k,K,m,n,T_0,T)$, such that
\begin{equation}
\label{eq:Solution_estimate_continuous}
\begin{aligned}
\|u\|_{C^{k,2+\alpha}_{WF}([T_0,T]\times \bar S_{n,m})} 
&\leq C\left(\|g\|_{C^{k,\alpha}_{WF}([0,T]\times\bar S_{n,m})}  + \|f\|_{C(\bar S_{n,m})}  \right).
\end{aligned}
\end{equation}
\end{thm}

\begin{rmk}[Coefficients of the operator $L$]
The coefficients of the operator $L$ are assumed to be functions only of the spatial variables, but it is straightforward to extend our results to time-dependent coefficients. Moreover the coefficients are assumed to be bounded functions. This restriction can be removed and replaced by a linear growth of the coefficients in the spatial variables, by incorporating a weight in the definition of the anisotropic H\"older spaces to take into account the growth of the coefficients, similarly to \cite[\S 2]{Feehan_Pop_mimickingdegen_pde}.
\end{rmk}

The proofs of Theorems \ref{thm:Time_interior_estimate} and \ref{thm:Interior_estimate} are based on a localization procedure described by N. V. Krylov in the proof of \cite[Theorem 8.11.1]{Krylov_LecturesHolder}. For this method to work, we need interpolation inequalities for our anisotropic H\"older spaces, which we establish in \S \ref{subsec:Interpolation_inequalities}, and we need global a priori Schauder estimates for model operators, which were established by C. Epstein and R. Mazzeo in \cite[Theorem 10.0.2]{Epstein_Mazzeo_annmathstudies}. These ideas are applicable to a more general functional analytical framework, where global a priori estimates and interpolation inequalities hold, and it was previously employed by P. Feehan and the author in the study of the regularity of solutions defined by a different class of degenerate elliptic equations with applications in Mathematical Finance \cite{Feehan_Pop_mimickingdegen_pde}.

\subsection{Comparison with previous research}
C. Epstein and R. Mazzeo prove in \cite[Corollary 3.2]{Epstein_Mazzeo_cont_est_diag} smoothness of solutions to the homogeneous initial-value problem defined by the operator $L$ with continuous, compactly supported data, under the assumption that the operator $L$ has a special diagonal structure, that is, the coefficients of the cross-terms in \eqref{eq:Generator} are $0$, and the drift coefficients $b_i(z)$ are bounded from below by a positive constant on $\bar S_{n,m}$. In Theorem \ref{thm:ExistenceUniqueness_continuous} we extend this result by not requiring any special structure of the operator $L$, other than the one implied by \eqref{eq:Generator} and by Assumption \ref{assump:Coeff}, and we prove local a priori Schauder estimates in Theorems \ref{thm:Time_interior_estimate} and \ref{thm:Interior_estimate}.

The results of \cite{Epstein_Mazzeo_cont_est_diag} are further extended in \cite[Theorem 1.1]{Epstein_Mazzeo_cont_est}, where the authors prove smoothness of solutions to the homogenous Kimura initial-value problem on compact manifolds with corners, $P$, when the initial data is assumed to belong to the weighted Sobolev space, $L^2(P, d\mu_L)$ (for the definition of the weight function $d\mu_L$ see \cite[\S 2]{Epstein_Mazzeo_cont_est}). The method of the proof of \cite{Epstein_Mazzeo_cont_est} is based on writing the Kimura operator in divergence form and applying the method of Moser iterations \cite{Moser_1964, Moser_1967, Moser_1971}. For this method to work, the authors prove that the weight function $d\mu_L$ is a doubling measure (\cite[Proposition 3.1]{Epstein_Mazzeo_cont_est}), and that a suitable $L^2$-invariant Poincar\'e inequality holds (\cite[Theorem 3.1]{Epstein_Mazzeo_cont_est}). As a consequence, it is established in \cite[Corollaries 4.1 and 4.2]{Epstein_Mazzeo_cont_est} that there is a H\"older exponent, $\alpha_0\in (0,1)$, such that the $C^{\alpha_0}_{WF}$-norm of the solution can be controlled in terms of its sup-norm. Comparing this result with our Theorem \ref{thm:Time_interior_estimate}, we prove that for \emph{all} $\alpha \in (0,1)$ and for all positive integers, $k$, the $C^{k,2+\alpha}_{WF}$-norm of the solution can be controlled in terms of the sup-norm of the initial data. In addition, our method of the proof appears to be more direct and it uses interpolation inequalities adapted to the anisotropic H\"older spaces (\S \ref{subsec:Interpolation_inequalities}) and a localization procedure due to N. V. Krylov (\cite{Krylov_LecturesHolder}).

\subsection{Applications of the main results}
We use the existence result in Theorem \ref{thm:ExistenceUniqueness} to establish the uniqueness in law and the strong Markov property of solutions to the standard Kimura stochastic differential equation and its singular drift perturbations in \cite[\S 2.3 and \S 3.2]{Pop_2013a}. The fact that we only require the coefficients of the operator $L$ to be H\"older continuous, allows us to also assume that the coefficients of the Kimura stochastic differential equation are only H\"older continuous, and so, our results in \cite[\S 2.3 and 3.2]{Pop_2013a} generalize the classical existence and uniqueness theorems of solutions to stochastic differential equations with Lipschitz continuous coefficients \cite[\S 5.2.B]{KaratzasShreve1991}. They also generalize the existence and uniqueness of weak solutions to a closely related degenerate stochastic differential equations studied in \cite{Athreya_Barlow_Bass_Perkins_2002, Bass_Perkins_2003}. Moreover the existence, uniqueness and the strong Markov property of weak solutions to Kimura stochastic differential equations and its singular drift perturbation are crucial ingredients in our proof of the Harnack inequality for nonnegative solutions to the homogeneous parabolic equation $u_t-Lu=0$, which we establish in forthcoming work joint with C. Epstein \cite[Theorem 7.6]{Epstein_Pop_2013b}.

\subsection{Outline of the article}
In \S \ref{sec:Holder_spaces}, building on the work of C. Epstein and R. Mazzeo \cite{Epstein_Mazzeo_annmathstudies}, we introduce anisotropic H\"older spaces adapted to our framework, and we prove interpolation inequalities for the new H\"older spaces in Proposition \ref{prop:InterpolationIneqS} and Corollary \ref{cor:Higher_order_interpolation_inequalities}. In \S \ref{sec:Time_interior_estimate} we begin by stating in Assumption \ref{assump:Coeff} the conditions satisfied by the coefficients of the operator $L$, and we then give the proofs of Theorems \ref{thm:Time_interior_estimate} and \ref{thm:Interior_estimate}. We prove Theorems \ref{thm:ExistenceUniqueness} and \ref{thm:ExistenceUniqueness_continuous} in \S \ref{sec:Existence_solutions}. In \S \ref{subsec:Notations}, we list the notations used in our article.

\subsection{Notations and conventions}
\label{subsec:Notations}
Let $\NN:=\{0,1,2,3,\ldots\}$. Given a positive integer $k$, we let $\NN^k$ denote the set of multi-indices $\alpha=(\alpha_1,\ldots,\alpha_k)\in\NN^k$, and we let $|\alpha|:=\alpha_1+\ldots+\alpha_k$. Given a finite set of elements, $F$, we let $|F|$ denote the cardinal of $F$. Let $B_r(z)$ denote the Euclidean ball centered at a point $z\in\bar S_{n,m}$ of radius $r$, relative to the domain $S_{n,m}$.

\subsection{Acknowledgment}
The author is indebted to Charles Epstein for suggesting this problem and for many very helpful discussions on this subject.

\section{Anisotropic H\"older spaces}
\label{sec:Holder_spaces}
In this section, we introduce the anisotropic H\"older spaces suitable for obtaining a priori Schauder estimates of solutions to the inhomogeneous initial-value problem \eqref{eq:Inhom_initial_value_problem}. The H\"older spaces defined in \S \ref{subsec:Holder_spaces} are a slight modification of the H\"older spaces introduced by C. Epstein and R. Mazzeo in their study of the existence, uniqueness and regularity of solutions to the parabolic problem defined by generalized Kimura operators \cite{Epstein_Mazzeo_2010, Epstein_Mazzeo_annmathstudies}. We then establish in \S \ref{subsec:Interpolation_inequalities} the interpolation inequalities satisfied by the anisotropic H\"older spaces. These properties will be a main ingredient in the proofs of the results in our article.

\subsection{Definition of the anisotropic H\"older spaces}
\label{subsec:Holder_spaces}
Following \cite[Chapter 5]{Epstein_Mazzeo_annmathstudies}, we need to first introduce a \emph{distance function}, $\rho$, which takes into account the degeneracy of the second-order coefficient matrix of the operator $L$. We let
\begin{equation}
\label{eq:rho}
\rho((t^0,z^0),(t,z)) := \rho_0(z^0,z) + \sqrt{|t^0-t|},\quad\forall\, (t^0,z^0), (t,z) \in [0,\infty)\times\bar S_{n,m},
\end{equation}
where $\rho_0$ is a distance function in the spatial variables. Because our domain $S_{n,m}$ is unbounded, as opposed to the compact manifolds considered in \cite{Epstein_Mazzeo_annmathstudies}, the properties of the distance function $\rho_0(z^0,z)$ depend on whether the points $z^0$ and $z$ are in a neighborhood of the boundary of $S_{n,m}$, or far away from the boundary of $S_{n,m}$. For any set of indices, $I\subseteq\{1,\ldots,n\}$, we let
\begin{align}
\label{eq:M_I}
M_I:=\left\{z=(x,y) \in S_{n,m}: x_i \in (0,1)\hbox{ for all } i \in I,\hbox{ and } x_j \in (1,\infty)\hbox{ for all } j \in I^c\right\},
\end{align}
where we denote $I^c:=\{1,2,\ldots,n\}\backslash I$. The distance function $\rho_0$ has the property that there is a positive constant, $c=c(n,m)$, such that for all sets of indices, $I, J\subseteq\{1,\ldots,n\}$, and all $z^0\in\bar M_I$ and $z\in\bar M_J$, we have that
\begin{equation}
\label{eq:Equivalent_intrinsic_metric}
\begin{aligned}
&c\left(\max_{i\in I\cap J} \left|\sqrt{x^0_i}-\sqrt{x_i}\right| + \max_{j\in (I\cap J)^c} |x^0_j-x_j| +\max_{l\in\{1,\ldots,m\}}|y^0_l-y_l|\right)\\
&\leq \rho(z^0,z)\\
&\leq c^{-1}\left(\max_{i\in I\cap J} \left|\sqrt{x^0_i}-\sqrt{x_i}\right| + \max_{j\in (I\cap J)^c} |x^0_j-x_j| +\max_{l\in\{1,\ldots,m\}}|y^0_l-y_l|\right).
\end{aligned}
\end{equation}
Let $k\in\NN$, $T>0$, and $U\subseteq S_{n,m}$. We let $C^k([0,T]\times U)$ denote the space consisting of functions $u:[0,T]\times U\rightarrow \RR$ that are continuous and locally bounded, and we let  $C^k([0,T]\times\bar U)$ denote the Banach space of functions $u:[0,T]\times\bar U\rightarrow \RR$, with continuous, bounded derivatives up to order $k$, endowed with the norm,
$$
\|u\|_{C^k([0,T]\times\bar U)} := \sum_{\stackrel{\tau\in\NN,\zeta\in\NN^{n+m}}{2\tau+|\zeta|\leq k}} \sup_{(t,z)\in [0,T]\times\bar U} |D^{\tau}_tD^{\zeta}_z u(t,z)|.
$$
We let $C^{\infty}([0,T]\times\bar U)$ be the space of smooth functions $u:[0,T]\times\bar U\rightarrow \RR$, with continuous and bounded derivatives of all orders, and we let $C^{\infty}_c([0,T]\times\bar U)$ be the space of smooth functions with compact support in $[0,T]\times\bar U$.

We recall the definition of the standard parabolic H\"older spaces \cite[\S 8.5]{Krylov_LecturesHolder}. Let $\alpha\in (0,1)$. Then $C^{0,\alpha}([0,T]\times\bar U)$ denotes the H\"older spaces of functions $u:[0,T]\times\bar U\rightarrow\RR$, such that
$$
\|u\|_{C^{0,\alpha}([0,T]\times\bar U)} := \|u\|_{C^0([0,T]\times\bar U)} 
+ \sup_{\stackrel{(t^0,z^0), (t,z)\in [0,T]\times \bar U}{(t^0,z^0)\neq (t,z)}} 
\frac{|u(t^0,z^0)-u(t,z)|}{\left(|z-z^0|+\sqrt{|t-t^0|}\right)^{\alpha}}.
$$
The space $C^{k,\alpha}([0,T]\times\bar U)$ consists of functions $u:[0,T]\times\bar U\rightarrow\RR$, such that for all $\tau\in\NN$ and $\zeta\in\NN^{n+m}$ satisfying the property that $2\tau+|\zeta| \leq k$, we have 
$$
D^{\tau}_t D^{\zeta}_z u \in C^{0,\alpha}([0,T]\times\bar U),
$$
and we endow the space $C^{k,\alpha}([0,T]\times\bar U)$ with the norm:
$$
\|u\|_{C^{k,\alpha}([0,T]\times\bar U)} := \sum_{\stackrel{\tau\in\NN, \zeta\in\NN^{n+m}}{2\tau+|\zeta|\leq k}}
 \|D^{\tau}_t D^{\zeta}_z\|_{C^{0,\alpha}([0,T]\times \bar U)}.
$$
Following \cite[\S 5.2.4]{Epstein_Mazzeo_annmathstudies}, we can now introduce the anisotropic H\"older spaces suitable to establish a priori Schauder estimates for solutions to the inhomogeneous initial-value problem \eqref{eq:Inhom_initial_value_problem}. We let $C^{0,\alpha}_{WF}([0,T]\times \bar U)$ be the H\"older space consisting of continuous functions, $u:[0,T]\times\bar U\rightarrow \RR$, such that the following norm is finite
$$
\|u\|_{C^{0,\alpha}_{WF}([0,T]\times \bar U)} := \|u\|_{C^0([0,T]\times\bar U)} 
+ \sup_{\stackrel{(t^0,z^0), (t,z)\in [0,T]\times \bar U}{(t^0,z^0)\neq (t,z)}} 
\frac{|u(t^0,z^0)-u(t,z)|}{\rho^{\alpha}((t^0,z^0), (t,z))}. 
$$
We let $C^{k,\alpha}_{WF}([0,T]\times \bar U)$ denote the H\"older space containing functions, $u\in C^k([0,T]\times\bar U)$, such that the derivatives $D^{\tau}_t D^{\zeta}_z$ belong to the space $C^{0,\alpha}_{WF}([0,T]\times \bar U)$, for all $\tau\in\NN$ and $\zeta\in\NN^{n+m}$, such that $2\tau+|\zeta| \leq k$. We endow the space $C^{k,\alpha}_{WF}([0,T]\times \bar U)$ with the norm,
\begin{align*}
\|u\|_{C^{k,\alpha}_{WF}([0,T]\times \bar U)} &:= \sum_{\stackrel{\tau\in\NN, \zeta\in\NN^{n+m}}{2\tau+|\zeta|\leq k}}
 \|D^{\tau}_t D^{\zeta}_z\|_{C^{0,\alpha}_{WF}([0,T]\times \bar U)}.
\end{align*}
We fix a set of indices, $I\subseteq \{1,\ldots, n\}$. Let $U$ be a set such that $U \subseteq M_{I}$. We let $C^{0,2+\alpha}_{WF}([0,T]\times \bar U)$ denote the H\"older space of functions, $u\in C^{1,\alpha}_{WF}([0,T]\times\bar U)\cap C^2([0,T]\times U)$, such that 
$$
u_t \in C^{0,\alpha}_{WF}([0,T]\times \bar U),
$$
and such that the functions,
\begin{align*}
\sqrt{x_ix_j}u_{x_ix_j}, \sqrt{x_i}u_{x_iy_l}, u_{y_ly_k} &\in C^{0,\alpha}_{WF}([0,T]\times \bar U),
\quad \forall\, i,j\in I,\quad \forall\, l,k=1,\ldots,m,\\
\sqrt{x_i}u_{x_i x_j}, u_{x_jx_k} &\in C^{0,\alpha}_{WF}([0,T]\times \bar U),\quad\forall\, i \in I, \quad\forall\, j,k \in I^c.
\end{align*}
We endowed the space $C^{0,2+\alpha}_{WF}([0,T]\times \bar U)$ with the norm,
\begin{align*}
\|u\|_{C^{0,2+\alpha}_{WF}([0,T]\times \bar U)} &:= \|u\|_{C^{1,\alpha}_{WF}([0,T]\times \bar U)} 
+ \sum_{i,j \in I}\|\sqrt{x_ix_j}u_{x_ix_j}\|_{C^{0,\alpha}_{WF}([0,T]\times \bar U)}\\
&\quad+ \sum_{l,k=1}^m\|u_{y_ly_k}\|_{C^{0,\alpha}_{WF}([0,T]\times \bar U)}
+ \sum_{i \in I} \sum_{j \in I^c}\|\sqrt{x_i}u_{x_ix_j}\|_{C^{0,\alpha}_{WF}([0,T]\times \bar U)}\\
&\quad+ \sum_{i \in I} \sum_{l=1}^m\|\sqrt{x_i}u_{x_iy_l}\|_{C^{0,\alpha}_{WF}([0,T]\times \bar U)}
+ \sum_{i, j \in I^c} \|u_{x_ix_j}\|_{C^{0,\alpha}_{WF}([0,T]\times \bar U)}\\
&\quad+ \sum_{i \in I^c} \sum_{l=1}^m \|u_{x_iy_l}\|_{C^{0,\alpha}_{WF}([0,T]\times \bar U)}
+\|u_t\|_{C^{0,\alpha}_{WF}([0,T]\times \bar U)} .
\end{align*}
We now consider the case when $U$ is an arbitrary set in $S_{n,m}$. We let $C^{0,2+\alpha}_{WF}([0,T]\times \bar U)$ denote the H\"older space consisting of functions $u\in C^2([0,T]\times U)$, satisfying the property that 
$$
u\upharpoonright_{\bar U\cap \bar M_{I}}\in C^{0,2+\alpha}_{WF}([0,T]\times(\bar U \cap \bar M_I)),\quad\forall\, I\subseteq\{1,\ldots,n\}.
$$
We endow the H\"older space $C^{0,2+\alpha}_{WF}([0,T]\times \bar U)$ with the norm
$$
\|u\|_{C^{0,2+\alpha}_{WF}([0,T]\times \bar U)} = \sum_{I \subseteq \{1,\ldots,n\}} \|u\|_{C^{0,2+\alpha}_{WF}([0,T]\times(\bar U \cap \bar M_I))}.
$$
We let $C^{k,2+\alpha}_{WF}([0,T]\times \bar U)$ be the space of functions $u\in C^k([0,T]\times U)$, satisfying the property that
$$
D^{\tau}_t D^{\zeta}_z u \in C^{0,2+\alpha}_{WF}([0,T]\times \bar U),\quad\forall\, \tau\in \NN,\forall\, \zeta\in\NN^{n+m}\hbox{ such that } 
2\tau+|\zeta| \leq k,
$$
and we endow it with the norm
\begin{align*}
\|u\|_{C^{k,2+\alpha}_{WF}([0,T]\times \bar U)} &:= \sum_{\stackrel{\tau\in \NN, \zeta\in\NN^{n+m}}{2\tau+|\zeta| \leq k}}
 \|D^{\tau}_t D^{\zeta}_z u\|_{C^{0,2+\alpha}_{WF}([0,T]\times \bar U)}.
\end{align*}
When $k=0$, we write for brevity $C([0,T]\times\bar U)$, $C^{\alpha}([0,T]\times \bar U)$, $C^{\alpha}_{WF}([0,T]\times \bar U)$ and $C^{2+\alpha}_{WF}([0,T]\times \bar U)$, instead of $C^0([0,T]\times\bar U)$, $C^{0,\alpha}([0,T]\times \bar U)$, $C^{0,\alpha}_{WF}([0,T]\times \bar U)$ and $C^{0,2+\alpha}_{WF}([0,T]\times \bar U)$.

The elliptic H\"older spaces $C^{k,\alpha}(\bar U)$, $C^{k,\alpha}_{WF}(\bar U)$ and $C^{k,2+\alpha}_{WF}(\bar U)$ are defined analogously to their parabolic counterparts, and so, we omit their definitions for brevity.

\subsection{Interpolation inequalities for anisotropic H\"older spaces}
\label{subsec:Interpolation_inequalities}
To prove the a priori Schauder estimates in Theorems \ref{thm:Time_interior_estimate} and \ref{thm:Interior_estimate}, and the existence and uniqueness of solutions in Theorems \ref{thm:ExistenceUniqueness} and \ref{thm:ExistenceUniqueness_continuous}, we need to develop suitable interpolation inequalities for the anisotropic H\"older spaces introduced in \S \ref{subsec:Holder_spaces}.  

For any set of indices, $I\subseteq\{1,\ldots,n\}$, we let
\begin{align}
\label{eq:M_I_prim}
M'_I&:=\left\{z=(x,y) \in S_{n,m}: x_i \in (0,1)\hbox{ for all } i \in I,\hbox{ and } x_j \in (1/2,\infty)\hbox{ for all } j \in I^c\right\},\\
\label{eq:M_I_secund}
M''_I&:=\left\{z=(x,y) \in S_{n,m}: x_i \in (0,2)\hbox{ for all } i \in I,\hbox{ and } x_j \in (1/4,\infty)\hbox{ for all } j \in I^c\right\},
\end{align}
where we recall that $I^c:=\{1,2,\ldots,n\}\backslash I$. Comparing the sets defined in \eqref{eq:M_I_prim} and \eqref{eq:M_I_secund} with the set $M_I$ defined in \eqref{eq:M_I}, we have that $M_I\subset M'_I\subset M''_I$.

We begin with

\begin{prop} [Interpolation inequalities]
\label{prop:InterpolationIneqS}
Let $T>0$ and $\alpha\in (0,1)$. Then there are positive constants, $C=C(\alpha, m, n, T)$ and $m_0=m_0(\alpha,m,n)$, such that for any function, $u \in C^{2+\alpha}_{WF}([0,T]\times\bar S_{n,m})$, and for all $\eps \in (0,1)$, the following hold:
\begin{align}
\label{eq:InterpolationIneqS1}
\|u\|_{C^{\alpha}_{WF}([0,T]\times\bar S_{n,m})} &\leq \eps \|u\|_{C^{2+\alpha}_{WF}([0,T]\times\bar S_{n,m})} 
+ C \eps^{-m_0} \|u\|_{C([0,T]\times\bar S_{n,m})},\\
\label{eq:InterpolationIneqS2}
\|u_{x_i}\|_{C([0,T]\times\bar S_{n,m})} &\leq \eps \|u\|_{C^{2+\alpha}_{WF}([0,T]\times\bar S_{n,m})} 
+ C \eps^{-m_0} \|u\|_{C([0,T]\times\bar S_{n,m})},\\
\label{eq:InterpolationIneqS2_prim}
\|u_{y_l}\|_{C([0,T]\times\bar S_{n,m})} &\leq \eps \|u\|_{C^{2+\alpha}_{WF}([0,T]\times\bar S_{n,m})} 
+ C \eps^{-m_0} \|u\|_{C([0,T]\times\bar S_{n,m})},\\
\label{eq:InterpolationIneqS2_secund}
\|u_t\|_{C([0,T]\times\bar S_{n,m})} &\leq \eps \|u\|_{C^{2+\alpha}_{WF}([0,T]\times\bar S_{n,m})} 
+ C \eps^{-m_0} \|u\|_{C([0,T]\times\bar S_{n,m})}.
\end{align}
Let $I\subseteq \{1,\ldots,n\}$, and assume in addition that the function $u$ has support in $[0,T]\times \bar M''_I$. Then, for all $l,k=1,\ldots,m$, the following hold:
\begin{align}
\label{eq:InterpolationIneqS4}
\|\sqrt{x_ix_j} u_{x_ix_j}\|_{C([0,T]\times\bar S_{n,m})} &\leq \eps \|u\|_{C^{2+\alpha}_{WF}([0,T]\times\bar S_{n,m})} 
+ C \eps^{-m_0} \|u\|_{C([0,T]\times\bar S_{n,m})},\quad\forall\, i,j\in I,\\
\label{eq:InterpolationIneqS4_prim}
\|\sqrt{x_i}u_{x_ix_j}\|_{C([0,T]\times\bar S_{n,m})} &\leq \eps \|u\|_{C^{2+\alpha}_{WF}([0,T]\times\bar S_{n,m})} 
+ C \eps^{-m_0} \|u\|_{C([0,T]\times\bar S_{n,m})},\quad\forall\, i \in I, j\in I^c,\\
\label{eq:InterpolationIneqS4_secund}
\|\sqrt{x_i}u_{x_iy_l}\|_{C([0,T]\times\bar S_{n,m})} &\leq \eps \|u\|_{C^{2+\alpha}_{WF}([0,T]\times\bar S_{n,m})} 
+ C \eps^{-m_0} \|u\|_{C([0,T]\times\bar S_{n,m})},\quad\forall\, i \in I,\\
\label{eq:InterpolationIneqS4_third}
\|u_{x_ix_j}\|_{C([0,T]\times\bar S_{n,m})} &\leq \eps \|u\|_{C^{2+\alpha}_{WF}([0,T]\times\bar S_{n,m})} 
+ C \eps^{-m_0} \|u\|_{C([0,T]\times\bar S_{n,m})},\quad\forall\, i,j \in I^c\\
\label{eq:InterpolationIneqS4_fourth}
\|u_{y_ly_k}\|_{C([0,T]\times\bar S_{n,m})} &\leq \eps \|u\|_{C^{2+\alpha}_{WF}([0,T]\times\bar S_{n,m})} 
+ C \eps^{-m_0} \|u\|_{C([0,T]\times\bar S_{n,m})},
\end{align}
and we also have that
\begin{align}
\label{eq:InterpolationIneqS3}
\|x_i u_{x_i}\|_{C^{\alpha}_{WF}([0,T]\times\bar S_{n,m})} &\leq \eps \|u\|_{C^{2+\alpha}_{WF}([0,T]\times\bar S_{n,m})} 
+ C \eps^{-m_0} \|u\|_{C([0,T]\times\bar S_{n,m})},\quad\forall\, i\in I,\\
\label{eq:InterpolationIneqS3_prim}
\|u_{x_j}\|_{C^{\alpha}_{WF}([0,T]\times\bar S_{n,m})} &\leq \eps \|u\|_{C^{2+\alpha}_{WF}([0,T]\times\bar S_{n,m})} 
+ C \eps^{-m_0} \|u\|_{C([0,T]\times\bar S_{n,m})},\quad\forall\, j\in I^c,\\
\label{eq:InterpolationIneqS3_secund}
\|u_{y_l}\|_{C^{\alpha}_{WF}([0,T]\times\bar S_{n,m})} &\leq \eps \|u\|_{C^{2+\alpha}_{WF}([0,T]\times\bar S_{n,m})} 
+ C \eps^{-m_0} \|u\|_{C([0,T]\times\bar S_{n,m})}.
\end{align}
\end{prop}

We give the technical proof of Proposition \ref{prop:InterpolationIneqS} at the end of the section.

\begin{rmk}[The hypothesis in Proposition \ref{prop:InterpolationIneqS} about the support of the function $u$]
To prove inequalities \eqref{eq:InterpolationIneqS4}-\eqref{eq:InterpolationIneqS3_secund}, we assume that the function $u$ has support in $[0,T]\times \bar M''_I$, for some set of indices, $I\subseteq\{1,\ldots,n\}$. This is only because of the fact that the weights of the derivatives in the $x_i$-coordinates of $u$ differ depending on whether the index $i$ belongs to the set $I$, or its complement, $I^c$, that is, the $x_i$-coordinate is small or large, respectively. Inequalities \eqref{eq:InterpolationIneqS4_fourth} and \eqref{eq:InterpolationIneqS3_secund} hold without the assumption that the support of $u$ is contained in $[0,T]\times \bar M''_I$.
\end{rmk}

\begin{rmk}[Comparison between the interpolation inequalities in the standard H\"older spaces and in the anisotropic H\"older spaces]
Notice that Proposition \ref{prop:InterpolationIneqS} does not establish the analogue of \cite[Inequality (8.8.4)]{Krylov_LecturesHolder}, that is,
\[
[u_{x_i}]_{C^{\alpha}([0,T]\times\bar S_{n,m})} \leq \eps \|u\|_{C^{2,\alpha}([0,T]\times\bar S_{n,m})} 
+ C \eps^{-m_0} \|u\|_{C([0,T]\times\bar S_{n,m})}.
\]
This is replaced by the weighted inequality \eqref{eq:InterpolationIneqS3}, due to the fact that the anisotropic H\"older space $C^{2+\alpha}_{WF}([0,T]\times\bar S_{n,m})$ allows for more general functions than the standard H\"older space $C^{2,\alpha}([0,T]\times\bar S_{n,m})$.
\end{rmk}

We have the following corollary to Proposition \ref{prop:InterpolationIneqS}, which contains the interpolation inequalities for the higher-order anisotropic H\"older spaces.
\begin{cor}[Higher-order interpolation inequalities]
\label{cor:Higher_order_interpolation_inequalities}
Let $\alpha\in (0,1)$, $k\in\NN$ and $T>0$. Then there are positive constants, $C=C(\alpha, k, m, n, T)$ and $m_k=m_k(\alpha,k,m,n)$, such that for any function, $u \in C^{k,2+\alpha}_{WF}([0,T]\times\bar S_{n,m})$, the following hold. Let $\tau\in\NN$ and $\zeta\in\NN^{n+m}$ be such that $2\tau+|\zeta|\leq k$, then for all $\eps \in (0,1)$, we have
\begin{align}
\label{eq:InterpolationIneqS1_higher_order}
\|D^{\tau}_tD^{\zeta}_zu\|_{C^{\alpha}_{WF}([0,T]\times\bar S_{n,m})} &\leq \eps \|u\|_{C^{k,2+\alpha}_{WF}([0,T]\times\bar S_{n,m})} 
+ C \eps^{-m_k} \|u\|_{C([0,T]\times\bar S_{n,m})},\\
\label{eq:InterpolationIneqS2_higher_order}
\|D^{\tau}_tD^{\zeta}_zu_{x_i}\|_{C([0,T]\times\bar S_{n,m})} &\leq \eps \|u\|_{C^{k,2+\alpha}_{WF}([0,T]\times\bar S_{n,m})} 
+ C \eps^{-m_k} \|u\|_{C([0,T]\times\bar S_{n,m})},\\
\label{eq:InterpolationIneqS2_prim_higher_order}
\|D^{\tau}_tD^{\zeta}_zu_{y_l}\|_{C([0,T]\times\bar S_{n,m})} &\leq \eps \|u\|_{C^{k,2+\alpha}_{WF}([0,T]\times\bar S_{n,m})} 
+ C \eps^{-m_k} \|u\|_{C([0,T]\times\bar S_{n,m})},\\
\label{eq:InterpolationIneqS2_secund_higher_order}
\|D^{\tau}_tD^{\zeta}_z u_t\|_{C([0,T]\times\bar S_{n,m})} &\leq \eps \|u\|_{C^{2+\alpha}_{WF}([0,T]\times\bar S_{n,m})} 
+ C \eps^{-m_k} \|u\|_{C([0,T]\times\bar S_{n,m})}.
\end{align}
Let $I\subseteq \{1,\ldots,n\}$, and assume in addition that the function $u$ has support in $[0,T]\times \bar M''_I$. Then, for all $l,p=1,\ldots,m$, the following hold
\begin{align}
\label{eq:InterpolationIneqS4_higher_order}
\|\sqrt{x_ix_j} D^{\tau}_tD^{\zeta}_zu_{x_ix_j}\|_{C([0,T]\times\bar S_{n,m})} &\leq \eps \|u\|_{C^{k,2+\alpha}_{WF}([0,T]\times\bar S_{n,m})} 
+ C \eps^{-m_k} \|u\|_{C([0,T]\times\bar S_{n,m})},\quad\forall\, i,j\in I,\\
\label{eq:InterpolationIneqS4_prim_higher_order}
\|\sqrt{x_i}D^{\tau}_tD^{\zeta}_zu_{x_ix_j}\|_{C([0,T]\times\bar S_{n,m})} &\leq \eps \|u\|_{C^{k,2+\alpha}_{WF}([0,T]\times\bar S_{n,m})} 
+ C \eps^{-m_k} \|u\|_{C([0,T]\times\bar S_{n,m})},\quad\forall\, i \in I, j\in I^c,\\
\label{eq:InterpolationIneqS4_secund_higher_order}
\|\sqrt{x_i}D^{\tau}_tD^{\zeta}_zu_{x_iy_l}\|_{C([0,T]\times\bar S_{n,m})} &\leq \eps \|u\|_{C^{k,2+\alpha}_{WF}([0,T]\times\bar S_{n,m})} 
+ C \eps^{-m_k} \|u\|_{C([0,T]\times\bar S_{n,m})},\quad\forall\, i \in I,\\
\label{eq:InterpolationIneqS4_third_higher_order}
\|D^{\tau}_tD^{\zeta}_zu_{x_ix_j}\|_{C([0,T]\times\bar S_{n,m})} &\leq \eps \|u\|_{C^{k,2+\alpha}_{WF}([0,T]\times\bar S_{n,m})} 
+ C \eps^{-m_k} \|u\|_{C([0,T]\times\bar S_{n,m})},\quad\forall\, i,j \in I^c,\\
\label{eq:InterpolationIneqS4_fourth_higher_order}
\|D^{\tau}_tD^{\zeta}_zu_{y_ly_p}\|_{C([0,T]\times\bar S_{n,m})} &\leq \eps \|u\|_{C^{k,2+\alpha}_{WF}([0,T]\times\bar S_{n,m})} 
+ C \eps^{-m_k} \|u\|_{C([0,T]\times\bar S_{n,m})},
\end{align}
and we also have that
\begin{align}
\label{eq:InterpolationIneqS3_higher_order}
\|x_i D^{\tau}_tD^{\zeta}_zu_{x_i}\|_{C^{\alpha}_{WF}([0,T]\times\bar S_{n,m})} &\leq \eps \|u\|_{C^{k,2+\alpha}_{WF}([0,T]\times\bar S_{n,m})} 
+ C \eps^{-m_k} \|u\|_{C([0,T]\times\bar S_{n,m})},\quad\forall\, i\in I,\\
\label{eq:InterpolationIneqS3_prim_higher_order}
\|D^{\tau}_tD^{\zeta}_zu_{x_j}\|_{C^{\alpha}_{WF}([0,T]\times\bar S_{n,m})} &\leq \eps \|u\|_{C^{k,2+\alpha}_{WF}([0,T]\times\bar S_{n,m})} 
+ C \eps^{-m_k} \|u\|_{C([0,T]\times\bar S_{n,m})},\quad\forall\, j\in I^c,\\
\label{eq:InterpolationIneqS3_secund_higher_order}
\|D^{\tau}_tD^{\zeta}_zu_{y_l}\|_{C^{\alpha}_{WF}([0,T]\times\bar S_{n,m})} &\leq \eps \|u\|_{C^{k,2+\alpha}_{WF}([0,T]\times\bar S_{n,m})} 
+ C \eps^{-m_k} \|u\|_{C([0,T]\times\bar S_{n,m})}.
\end{align}
\end{cor}

\begin{proof}
From the definition of the anisotropic H\"older spaces in \S \ref{subsec:Holder_spaces}, the function $D^{\tau}_tD^{\zeta}_zu$ belongs to the H\"older space $C^{0,2+\alpha}_{WF}([0,T]\times\bar U)$, for all $\tau\in\NN$, $\zeta\in\NN^{n+m}$ with the property that $2\tau+|\zeta|\leq k$, and so, Proposition \ref{prop:InterpolationIneqS} applies to the function $D^{\tau}_tD^{\zeta}_zu$. Thus, it is sufficient to show that inequality \eqref{eq:InterpolationIneqS1_higher_order} holds, because the rest of the interpolation inequalities, \eqref{eq:InterpolationIneqS2_higher_order}-\eqref{eq:InterpolationIneqS3_secund_higher_order}, follow by applying Proposition \ref{prop:InterpolationIneqS} to the function $D^{\tau}_tD^{\zeta}_zu$ instead of $u$, and using inequality \eqref{eq:InterpolationIneqS1_higher_order}. We now prove inequality \eqref{eq:InterpolationIneqS1_higher_order}.

Applying inequality \eqref{eq:InterpolationIneqS1} to the function $D^{\tau}_tD^{\zeta}_z u$ instead of $u$, and then applying  inequalities \eqref{eq:InterpolationIneqS2}, \eqref{eq:InterpolationIneqS2_prim} and \eqref{eq:InterpolationIneqS2_secund} to the derivatives of the function $u$, we obtain that there are positive constants, $C=C(\alpha,k,m,n,T)$ and $m_k=m_k(\alpha,k,m,n)$, such that for all $\eps\in (0,1)$, inequality \eqref{eq:InterpolationIneqS1} holds. This completes the proof.
\end{proof}

We use Corollary \ref{cor:Higher_order_interpolation_inequalities} to prove the following Lemmas \ref{lem:Estimate_Lu} and \ref{lem:Estimate_varphi_Lu}. Both Lemmas \ref{lem:Estimate_Lu} and \ref{lem:Estimate_varphi_Lu} are technical estimates used in the proofs of Theorems \ref{thm:Time_interior_estimate} and \ref{thm:ExistenceUniqueness}, respectively.

\begin{lem}[Estimate of $Lu$]
\label{lem:Estimate_Lu}
Let $\alpha\in (0,1)$, $k\in\NN$ and $T>0$. Let $I\subseteq\{1,\ldots,n\}$ and let $U$ be an open set in $M''_I$. Suppose that the coefficients of the operator $L$ satisfy property \eqref{eq:Holder_cont}. Then there are positive constants, $C=C(\alpha,k,K,m,n)$ and $m_k=m_k(\alpha,k,m,n)$, such that for all functions $u\in C^{k,2+\alpha}_{WF}([0,T]\times\bar S_{n,m})$ with support in $[0,T]\times\bar U$, we have that
\begin{equation}
\label{eq:Estimate_Lu}
\|Lu\|_{C^{k,\alpha}_{WF}([0,T]\times\bar U)} \leq (\Lambda+C\eps)\|u\|_{C^{k,2+\alpha}_{WF}([0,T]\times\bar U)}
 + C\eps^{-m_k} \|u\|_{C([0,T]\times\bar U)}.
\end{equation}
where the positive constant $\Lambda$ is given by
\begin{equation}
\label{eq:Constant_estimate_Lu}
\begin{aligned}
\Lambda&:=\sum_{i\in I}\|a_{ii}\|_{C(\bar U)}+\sum_{i\in I^c}\|x_ia_{ii}\|_{C(\bar U)} + \sum_{i,j\in I}\|\tilde a_{ij}\|_{C(\bar U)}\\
&\quad+\sum_{i\in I,\, j\in I^c}\|x_j\tilde a_{ij}\|_{C(\bar U)}+ \sum_{j\in I,\, i\in I^c}\|x_i\tilde a_{ij}\|_{C(\bar U)}+\sum_{i=1}^n\|b_i\|_{C(\bar U)}\\
&\quad+ \sum_{i\in I}\sum_{l=1}^m\|c_{il}\|_{C(\bar U)}+ \sum_{i\in I^c}\sum_{l=1}^m\|x_ic_{il}\|_{C(\bar U)}  +\sum_{l,p=1}^m\|d_{lp}\|_{C(\bar U)} + +\sum_{l=1}^m\|e_l\|_{C(\bar U)}. 
\end{aligned}
\end{equation}
\end{lem}

\begin{rmk}
\label{rmk:Reduction_to_one_set_of_indices}
Given a set of indices $I\subseteq\{1,\ldots,n\}$ and a function $u$ with support in $[0,T]\times\bar M''_I$, notice that there is no loss of generality in assuming that $I=\{1,\ldots,n\}$, because for any other set $I\neq \{1,\ldots,n\}$, the $x_j$-variables, for $j\in I^c $, can be treated as the $y_l$-variables, for all $l=1,\ldots,m$. That is, by relabeling the variables, we may replace the space $S_{n,m}$ by $S_{n',m'}$, where $n'=|I|$ and $m'=m+n-|I|$, so that the support of the function $u$ becomes a subset of $\bar M''_{\{1,\ldots,n'\}}$. Note that $n'$ may be zero.
\end{rmk}

\begin{proof}[Proof of Lemma \ref{lem:Estimate_Lu}]
Remark \ref{rmk:Reduction_to_one_set_of_indices} shows that we may assume without loss of generality that $U\subseteq M''_{\{1,\ldots,n\}}$. Under the assumption that $U\subseteq M''_{\{1,\ldots,n\}}$, the definition of the constant $\Lambda$ in \eqref{eq:Constant_estimate_Lu} simplifies to
\begin{equation}
\label{eq:Constant_estimate_Lu_aux}
\begin{aligned}
\Lambda&:=\sum_{i=1}^n\|a_{ii}\|_{C(\bar U)} + \sum_{i,j=1}^n\|\tilde a_{ij}\|_{C(\bar U)}+\sum_{i=1}^n\|b_i\|_{C(\bar U)}\\
&\quad+ \sum_{i=1}^n\sum_{l=1}^m\|c_{il}\|_{C(\bar U)}  +\sum_{l,p=1}^m\|d_{lp}\|_{C(\bar U)} +\sum_{l=1}^m\|e_l\|_{C(\bar U)}. 
\end{aligned}
\end{equation}
Let $\tau\in\NN$ and $\zeta\in \NN^{n+m}$ be such that $2\tau+|\zeta| \leq k$. We need to show that the functions $D^{\tau}_tD^{\zeta}_z(Lu)$ have $C^{\alpha}_{WF}([0,T]\times\bar U)$-H\"older norm bounded by the right-hand side of inequality \eqref{eq:Estimate_Lu}. That is, for all $i,j=1,\ldots,n$ and all $l,p=1,\ldots,m$, the functions
\begin{equation}
\label{eq:All_functions}
\begin{aligned}
&D^{\tau}_tD^{\zeta}_z(x_ia_{ii}u_{x_ix_i}),\quad D^{\tau}_tD^{\zeta}_z(x_ix_j\tilde a_{ij}u_{x_ix_j}),\quad D^{\tau}_tD^{\zeta}_z(b_iu_{x_i}),\\
&D^{\tau}_tD^{\zeta}_z(x_ic_{il}u_{x_iy_l}),\quad D^{\tau}_tD^{\zeta}_z(d_{lp}u_{y_ly_p}),\quad D^{\tau}_tD^{\zeta}_z(e_lu_{y_l})
\end{aligned}
\end{equation}
have $C^{\alpha}_{WF}([0,T]\times\bar U)$-H\"older norms bounded by the right-hand side of inequality \eqref{eq:Estimate_Lu}. We will prove this fact for the function $D^{\tau}_tD^{\zeta}_z(x_i a_{ii}u_{x_ix_i})$, but the analysis is the same for all the remaining functions in \eqref{eq:All_functions}, and so, we omit the details for brevity. Direct calculations give us
\begin{equation}
\label{eq:Expansion_derivative}
\begin{aligned}
D^{\tau}_tD^{\zeta}_z(x_ia_{ii}u_{x_ix_i})
&= x_ia_{ii}(z)\left(D^{\tau}_tD^{\zeta}_z u\right)_{x_ix_i} + \mathbf{1}_{\{\zeta_i\geq 1\}}a_{ii}(z)\left(D^{\tau}_tD^{\zeta-e_i}_z u\right)_{x_ix_i}\\ 
&\quad+ \mathbf{1}_{\{\zeta_i\geq 1\}} \sum_{\stackrel{\zeta'+\zeta''=\zeta-e_i}{|\zeta''|\leq |\zeta|-2}} D^{\zeta'}_z a_{ii}(z) \left(D^{\tau}_tD^{\zeta''}_z u\right)_{x_ix_i}\\
&\quad+x_i \sum_{\stackrel{\zeta'+\zeta''=\zeta}{|\zeta'|\geq 1}} \left(D^{\zeta'}_z a_{ii}\right)(z) \left(D^{\tau}_tD^{\zeta''}_z u\right)_{x_ix_i},
\end{aligned}
\end{equation}
where $\zeta'$ and $\zeta''$ are multi-indices in $\NN^{n+m}$. We now estimate each term of the preceding inequality.

Applying \cite[Inequality (5.62)]{Epstein_Mazzeo_annmathstudies}, we have that
\begin{align*}
\|x_ia_{ii}D^{\tau}_tD^{\zeta}_z u_{x_ix_i}\|_{C^{\alpha}_{WF}([0,T]\times\bar U)}
&\leq \Lambda \|x_iD^{\tau}_tD^{\zeta}_z u_{x_ix_i}\|_{C^{\alpha}_{WF}([0,T]\times\bar U)}\\
&\quad+ K \|x_iD^{\tau}_tD^{\zeta}_z u_{x_ix_i}\|_{C([0,T]\times\bar U)},
\end{align*}
where we recall the definitions of the positive constants $\Lambda$ and $K$ in \eqref{eq:Constant_estimate_Lu_aux} and \eqref{eq:Holder_cont}, respectively. Because $2\tau+|\zeta| \leq k$, and we assume that $u \in C^{k,2+\alpha}_{WF}([0,T]\times\bar U)$, we have that
$$
\|x_iD^{\tau}_tD^{\zeta}_z u_{x_ix_i}\|_{C^{\alpha}_{WF}([0,T]\times\bar U)} \leq \|u\|_{C^{k,2+\alpha}_{WF}([0,T]\times\bar U)},
$$
and applying the interpolation inequality \eqref{eq:InterpolationIneqS4_higher_order}, we also have that there are positive constants, $C=C(\alpha,k,m,n,T)$ and $m_k=m_k(\alpha,k,m,n)$, such that
$$
\|x_iD^{\tau}_tD^{\zeta}_z u_{x_ix_i}\|_{C([0,T]\times\bar U)} \leq \eps\|u\|_{C^{k,2+\alpha}_{WF}([0,T]\times\bar U)}
+ C\eps^{-m_k} \|u\|_{C([0,T]\times\bar U)} .
$$
Combining the preceding three inequalities, we obtain that
\begin{equation}
\label{eq:Expansion_derivative_term1}
\begin{aligned}
\|x_ia_{ii} D^{\tau}_tD^{\zeta}_zu_{x_ix_i}\|_{C^{\alpha}_{WF}([0,T]\times\bar U)}
&\leq  (\Lambda+K\eps)\|u\|_{C^{k,2+\alpha}_{WF}([0,T]\times\bar U)}  + C\eps^{-m_k}  \|u\|_{C([0,T]\times\bar U)}, 
\end{aligned}
\end{equation}
where now $C=C(\alpha,k,K,m,n,T)$ is a positive constant.

Applying again \cite[Inequality (5.62)]{Epstein_Mazzeo_annmathstudies}, we have that
\begin{align*}
\|a_{ii}D^{\tau}_tD^{\zeta-e_i}_z u_{x_ix_i}\|_{C^{\alpha}_{WF}([0,T]\times\bar U)}
&\leq \Lambda \|D^{\tau}_tD^{\zeta-e_i}_z u_{x_ix_i}\|_{C^{\alpha}_{WF}([0,T]\times\bar U)}\\
&\quad+ K \|D^{\tau}_tD^{\zeta-e_i}_z u_{x_ix_i}\|_{C([0,T]\times\bar U)}.
\end{align*}
Writing the function $D^{\tau}_tD^{\zeta-e_i}_z u_{x_ix_i} = D^{\tau}_tD^{\zeta}_z u_{x_i}$, and using the fact that $2\tau+|\zeta| \leq k$ and $u \in C^{k,2+\alpha}_{WF}([0,T]\times\bar U)$, the interpolation inequality \eqref{eq:InterpolationIneqS2_higher_order}
gives us
\begin{equation*}
\begin{aligned}
\| D^{\tau}_tD^{\zeta}_z u_{x_i}\|_{C([0,T]\times\bar U)}
&\leq  \eps\|u\|_{C^{k,2+\alpha}_{WF}([0,T]\times\bar U)}  + C\eps^{-m_k}  \|u\|_{C([0,T]\times\bar U)}, 
\end{aligned}
\end{equation*}
and so, we obtain
\begin{equation}
\label{eq:Expansion_derivative_term2}
\begin{aligned}
\| a_{ii}D^{\tau}_tD^{\zeta-e_i}_z u_{x_ix_i}\|_{C^{\alpha}_{WF}([0,T]\times\bar U)}
&\leq  (\Lambda+K\eps)\|u\|_{C^{k,2+\alpha}_{WF}([0,T]\times\bar U)}  + C\eps^{-m_k}  \|u\|_{C([0,T]\times\bar U)}, 
\end{aligned}
\end{equation}
We now consider the case of multi-indices $\zeta'$ and $\zeta''$ in $\NN^{n+m}$, such that $\zeta'+\zeta''=\zeta-e_i$ and $|\zeta''|\leq |\zeta|-2$. Because $2\tau+|\zeta| \leq k$, then $2\tau+|\zeta''|+2\leq k$, and using the fact that $u \in C^{k,2+\alpha}_{WF}([0,T]\times\bar U)$, we may apply the the interpolation inequality \eqref{eq:InterpolationIneqS1_higher_order} to the function $D^{\tau}_tD^{\zeta''}_z u_{x_ix_i}$, together with \cite[Inequality (5.62)]{Epstein_Mazzeo_annmathstudies}, to obtain 
\begin{equation}
\label{eq:Expansion_derivative_term3}
\begin{aligned}
\| D^{\zeta'}_za_{ii}D^{\tau}_tD^{\zeta''}_z u_{x_ix_i}\|_{C^{\alpha}_{WF}([0,T]\times\bar U)}
&\leq  (\Lambda+K\eps)\|u\|_{C^{k,2+\alpha}_{WF}([0,T]\times\bar U)}  + C\eps^{-m_k}  \|u\|_{C([0,T]\times\bar U)}, 
\end{aligned}
\end{equation}
It remains to consider the case of indices $\zeta'$ and $\zeta''$ in $\NN^{n+m}$, such that $\zeta'+\zeta''=\zeta$ and $|\zeta'|\geq 1$. Then we have that $2\tau+|\zeta''|+2\leq k+1$, and we may apply the interpolation inequalities \eqref{eq:InterpolationIneqS3_higher_order} to the function $D^{\tau}_tD^{\zeta''}_z u_{x_ix_i}$, together with \cite[Inequality (5.62)]{Epstein_Mazzeo_annmathstudies}, to obtain
\begin{equation}
\label{eq:Expansion_derivative_term4}
\begin{aligned}
\| x_iD^{\zeta'}_za_{ii}D^{\tau}_tD^{\zeta''}_z u_{x_ix_i}\|_{C^{\alpha}_{WF}([0,T]\times\bar U)}
&\leq  (\Lambda+K\eps)\|u\|_{C^{k,2+\alpha}_{WF}([0,T]\times\bar U)} \\
&\quad + C\eps^{-m_k}  \|u\|_{C([0,T]\times\bar U)}, 
\end{aligned}
\end{equation}
Combining identity \eqref{eq:Expansion_derivative} with inequalities \eqref{eq:Expansion_derivative_term1}-\eqref{eq:Expansion_derivative_term4}, we obtain that the $C^{\alpha}_{WF}([0,T]\times\bar U)$-H\"older norm of the function $D^{\tau}_tD^{\zeta}_z(x_ia_{ii}u_{x_ix_i})$ is controlled by the right-hand side of inequality \eqref{eq:Estimate_Lu}. This completes the proof.
\end{proof}

We have another technical estimate used in the proof of Theorem \ref{thm:ExistenceUniqueness}.
\begin{lem}[Estimate of $\varphi Lu$]
\label{lem:Estimate_varphi_Lu}
Let $\alpha\in (0,1)$, $k\in\NN$ and $T>0$. Let $I\subseteq\{1,\ldots,n\}$ and let $\varphi\in C^{k,\alpha}_{WF}(\bar S_{n,m})$ be a function with support in $\bar M'_I$. Suppose that the coefficients of the operator $L$ satisfy property \eqref{eq:Holder_cont}. Then there are positive constants, $C=C(\alpha,\|\varphi\|_{C^{k,\alpha}_{WF}(\bar S_{n,m})},k,K,m,n,T)$ and $m_k=m_k(\alpha,k,m,n)$, such that for all functions $u\in C^{k,2+\alpha}_{WF}([0,T]\times\bar S_{n,m})$ with support in $[0,T]\times\bar M''_I$, we have that
\begin{equation}
\label{eq:Estimate_varphi_Lu}
\begin{aligned}
\|\varphi Lu\|_{C^{k,\alpha}_{WF}([0,T]\times\bar S_{n,m})} 
&\leq \left(\Lambda\|\varphi\|_{C(\bar S_{n,m})} +C\eps\right)\|u\|_{C^{k,2+\alpha}_{WF}([0,T]\times\bar S_{n,m})}\\
&\quad + C\eps^{-m_k} \|u\|_{C([0,T]\times\bar S_{n,m})},
\end{aligned}
\end{equation}
where the positive constant $\Lambda$ is given by \eqref{eq:Constant_estimate_Lu} with the set $\bar U$ replaced by $\supp\varphi$.
\end{lem}

\begin{proof}
Estimate \eqref{eq:Estimate_varphi_Lu} if a straightforward consequence of Lemma \ref{lem:Estimate_Lu}. We only need to change the coefficients of the operator $L$ by multiplying them by the function $\varphi$. We omit the detailed proof.
\end{proof}

We now give the proof of the interpolation inequalities in the anisotropic H\"older spaces.
\begin{proof} [Proof of Proposition \ref{prop:InterpolationIneqS}]
We consider $\eta\in(0,1)$ to be a suitably chosen constant during the proofs of each of inequalities \eqref{eq:InterpolationIneqS1}-\eqref{eq:InterpolationIneqS3_secund}. We divide the proof into several steps.

\begin{step}[Proof of inequality \eqref{eq:InterpolationIneqS1}]
\label{step:InterpolationIneqS1}
It is sufficient to show that inequality \eqref{eq:InterpolationIneqS1} holds for the seminorm $[u]_{C^{\alpha}_{WF}([0,T]\times\bar S_{n,m})}$, and for this purpose we only need to consider differences, $u(P_1)-u(P_2)$, where all except one of the coordinates of the points $P_1, P_2 \in [0,T]\times\bar S_{n,m}$ are identical. We outline the proof when the $x_i$-coordinates of $P_1$ and $P_2$ differ, but the case of the $t$-coordinate and of the $y_l$-coordinates can be treated in a similar manner. Notice that from inequalities \eqref{eq:Equivalent_intrinsic_metric}, we can find a positive constant, $C$, such that
\begin{equation}
\label{eq:Quotient_Euclidean_rho_distance}
\frac{|x^1_i-x^2_i|}{\rho(P_1,P_2)} \leq C,
\end{equation}
We consider two situations: $|x^1_i-x^2_i|\leq \eta$ and $|x^1_i-x^2_i|> \eta$.
\begin{case}[Points with $x_i$-coordinates close together]
\label{case:InterpolationIneqS1_case1}
Assuming that $|x^1_i-x^2_i|\leq \eta$, we have
\begin{equation*}
\begin{aligned}
|u(P_1)-u(P_2)| &\leq |x^1_i-x^2_i| \|u_{x_i}\|_{C([0,T]\times\bar S_{n,m})} \\
                &\leq \eta \frac{|x^1_i-x^2_i|}{\eta} \|u\|_{C^{2+\alpha}_{WF}([0,T]\times\bar S_{n,m})}\\
                &\leq \eta \left(\frac{|x^1_i-x^2_i|}{\eta} \right)^{\alpha} \|u\|_{C^{2+\alpha}_{WF}([0,T]\times\bar S_{n,m})}\\
                &\leq \eta^{1-\alpha} \left(\frac{|x^1_i-x^2_i|}{\rho(P_1,P_2)}\right)^{\alpha} \rho^{\alpha}(P_1,P_2)\|u\|_{C^{2+\alpha}_{WF}([0,T]\times\bar S_{n,m})},
\end{aligned}
\end{equation*}
and so, using inequality \eqref{eq:Quotient_Euclidean_rho_distance}, there is a positive constant, $C=C(\alpha)$, such that
\begin{equation}
\label{eq:InterpolationIneqS1_2}
\begin{aligned}
\frac{|u(P_1)-u(P_2)|}{\rho^{\alpha}(P_1,P_2)} \leq C\eta^{1-\alpha}  \|u\|_{C^{2+\alpha}_{WF}([0,T]\times\bar S_{n,m})},
\end{aligned}
\end{equation}
which concludes this case.
\end{case}

\begin{case}[Points with $x_i$-coordinates farther apart]
\label{case:InterpolationIneqS1_case2}
Assuming that $|x^1_i-x^2_i|> \eta$, we have
$$
1<\left(\frac{|x^1_i-x^2_i|}{\eta} \right)^{\alpha} = \eta^{-\alpha}\left(\frac{|x^1_i-x^2_i|}{\rho(P_1,P_2)}\right)^{\alpha} \rho^{\alpha}(P_1,P_2),
$$
from where it follows by inequality \eqref{eq:Quotient_Euclidean_rho_distance} that there is a positive constant, $C=C(\alpha)$,  such that $1 \leq C\eta^{-\alpha}  \rho^{\alpha}(P_1,P_2)$. We then obtain that
$$
|u(P_1)-u(P_2)| \leq 2 \|u\|_{C([0,T]\times\bar S_{n,m})} \leq  C\eta^{-\alpha} \rho^{\alpha}(P_1,P_2) \|u\|_{C([0,T]\times\bar S_{n,m})},
$$
which is equivalent to
\begin{equation}
\label{eq:InterpolationIneqS1_3}
\begin{aligned}
\frac{|u(P_1)-u(P_2)|}{\rho^{\alpha}(P_1,P_2)} &\leq C\eta^{-\alpha} \|u\|_{C([0,T]\times\bar S_{n,m})},
\end{aligned}
\end{equation}
which concludes this case.
\end{case}
By combining inequalities \eqref{eq:InterpolationIneqS1_2} and \eqref{eq:InterpolationIneqS1_3}, we obtain
$$
[u]_{C^{\alpha}_{WF}([0,T]\times\bar S_{n,m})} \leq C \eta^{1-\alpha} \|u\|_{C^{2+\alpha}_{WF}([0,T]\times\bar S_{n,m})} + C \eta^{-\alpha} \|u\|_{C([0,T]\times\bar S_{n,m})}.
$$
Since $\eps\in(0,1)$, we may choose $\eta\in(0,1)$ such that $\eps=C\eta^{1-\alpha}$. The preceding inequality then gives \eqref{eq:InterpolationIneqS1}.
\end{step}

Inequalities \eqref{eq:InterpolationIneqS2}, \eqref{eq:InterpolationIneqS2_prim} and \eqref{eq:InterpolationIneqS2_secund} are proved using a similar method, and so, we only outline the proof of inequality \eqref{eq:InterpolationIneqS2}.
\begin{step}[Proof of inequality \eqref{eq:InterpolationIneqS2}]
\label{step:InterpolationIneqS2}
Let $P\in [0,T]\times\bar S_{n,m}$. Then, for any $\eta>0$, we have
\begin{align*}
|u_{x_i}(P)| &\leq \left|u_{x_i}(P)-\eta^{-1} \left(u(P+\eta e_i) - u(P)\right)\right| + 2 \eta^{-1} \|u\|_{C([0,T]\times\bar S_{n,m})}\\
             &=    |u_{x_i}(P)- u_{x_i}(P+\eta\theta e_i)| + 2 \eta^{-1} \|u\|_{C([0,T]\times\bar S_{n,m})}\\
             &=    \frac{|u_{x_i}(P)- u_{x_i}(P+\eta\theta e_i)|}
             {\rho^{\alpha}(P,P+\eta\theta e_i)}\rho^{\alpha}(P,P+\eta\theta e_i)
                    + 2 \eta^{-1} \|u\|_{C([0,T]\times\bar S_{n,m})},
\end{align*}
for some constant $\theta \in [0,1]$. From inequality \eqref{eq:Equivalent_intrinsic_metric} and the fact that we choose $\eta\in (0,1)$, we obtain that there is a positive constant, $C$, such that
\begin{equation}
\label{eq:ADDS}
\rho(P,P+\eta\theta e_i) \leq  C\eta^{1/2},\quad \forall\, P\in [0,T]\times\bar S_{n,m},
\end{equation}
which gives us
\begin{equation*}
|u_{x_i}(P)| \leq \eta^{\alpha/2}[u_{x_i}]_{C^{\alpha}_{WF}([0,T]\times\bar S_{n,m})}
                    + 2 \eta^{-1} \|u\|_{C([0,T]\times\bar S_{n,m})}, \quad \forall\, P\in [0,T]\times\bar S_{n,m}.
\end{equation*}
Since $\eps\in(0,1)$, we may choose $\eta\in(0,1)$ such that $\eps=\eta^{\alpha/2}$, and inequality \eqref{eq:InterpolationIneqS2} follows immediately from the preceding one.
\end{step}

For inequalities \eqref{eq:InterpolationIneqS4}-\eqref{eq:InterpolationIneqS3_secund}, we assume that the function $u$ is supported in $[0,T]\times\bar M''_I$, for a set of indices $I\subseteq\{1,\ldots,n\}$. From Remark \ref{rmk:Reduction_to_one_set_of_indices}, without loss of generality we may assume that $I=\{1,2,\ldots,n\}$, because otherwise the $x_j$-variables, for $j\in I^c $, can be treated as the $y_l$-variables, for all $l=1,\ldots,m$. That is, by relabeling the variables, we may replace the space $S_{n,m}$ by $S_{n',m'}$, where $n'=|I|$ and $m'=m+n-|I|$, so that the function $u$ may be assumed to have support in $[0,T]\times \bar M''_{\{1,\ldots,n'\}}$. 

Note that it is sufficient to prove inequalities \eqref{eq:InterpolationIneqS4}, \eqref{eq:InterpolationIneqS4_prim} and \eqref{eq:InterpolationIneqS4_fourth}, as inequalities \eqref{eq:InterpolationIneqS4_secund} and \eqref{eq:InterpolationIneqS4_third} follow from \eqref{eq:InterpolationIneqS4_prim} and \eqref{eq:InterpolationIneqS4_fourth}, respectively. The proof of inequality \eqref{eq:InterpolationIneqS4_fourth} is very similar to the proofs of inequalities \eqref{eq:InterpolationIneqS4} and \eqref{eq:InterpolationIneqS4_prim}, but simpler, and so, we omit its detailed proof for brevity.

\begin{step}[Proof of inequalities \eqref{eq:InterpolationIneqS4} and \eqref{eq:InterpolationIneqS4_prim}]
For any point, $P=(t,z)\in [0,T]\times\bar S_{n,m}$, and $\eta>0$, we have for all $i\in I$ and $j=1,\ldots,n$,
\begin{align}
|u_{x_ix_j}(P)|
&\leq\left|u_{x_ix_j}(P) - \eta^{-1}\left(u_{x_i}(P+\eta e_j) -u_{x_i}(P) \right)\right|
+ \eta^{-1}\left(|u_{x_i}(P)| + |u_{x_i}(P+ \eta e_j)|\right)\notag\\
&\label{eq:ADD1}
\leq  \left|u_{x_ix_j}(P) - u_{x_ix_j}(P+\theta \eta e_j) \right| + \eta^{-1}\left(|u_{x_i}(P)| + |u_{x_i}(P+ \eta e_j)|\right),
\end{align}
for some $\theta \in [0,1]$. If $j \in I^c$, we have
\begin{equation*}
\begin{aligned}
|\sqrt{x_i}u_{x_ix_j}(P)|
&\leq\frac{|\sqrt{x_i}u_{x_ix_j}(P) - \sqrt{x_i}u_{x_ix_j}(P+\theta \eta e_j)|}{\rho^{\alpha}(P,P+\theta\eta e_j)}
                          \rho^{\alpha}(P,P+\theta\eta e_j)\\
&\quad + 2 \eta^{-1}\|\sqrt{x_i}u_{x_i}\|_{C([0,T]\times\bar S_{n,m})} \\
&\leq  C \eta^{\alpha/2}[\sqrt{x_i}u_{x_ix_j}]_{C^{\alpha}_{WF}([0,T]\times\bar S_{n,m})}
                             + 2 \eta^{-1}\|\sqrt{x_i}u_{x_i}\|_{C([0,T]\times\bar S_{n,m})} \hbox{  (by \eqref{eq:ADDS}).}
\end{aligned}
\end{equation*}
For all $\eps\in(0,1)$, we may choose $\eta\in(0,1)$ such that $\eps=C\eta^{\alpha/2}$ in the preceding inequality. Combining the resulting inequality with \eqref{eq:InterpolationIneqS2}, and using the fact that the function $u$ has support in $
[0,T]\times\bar M''_I$, and that the domain $M''_I$ is bounded in the $x_i$-coordinate, for all $i\in I$, we see that estimate \eqref{eq:InterpolationIneqS4_prim} holds for all $j \in I^c$.

Next, we consider the case when $j \in I$, that is, we want to prove inequality \eqref{eq:InterpolationIneqS4}. For brevity, we denote $P'=P+\theta\eta e_j$. We consider two distinct cases depending on whether $\eta < x'_j/2$ or $\eta \geq x'_j/2$. 

\setcounter{case}{0}
\begin{case}[Points with $x_j$-coordinates small] Assuming that $\eta < x'_j/2$, we obtain by \eqref{eq:ADD1} that
\begin{equation}
\label{eq:InterpolationIneq4_11}
\begin{aligned}
\left|\sqrt{x_ix_j}u_{x_ix_j}(P)\right| 
&\leq  \frac{\left|\sqrt{x_ix_j}u_{x_ix_j}(P) - \sqrt{x_ix'_j}u_{x_ix_j}(P')\right|}{\rho^{\alpha}(P,P')} \rho^{\alpha}(P,P')\\
&\quad + \left|\left(\sqrt{x_ix_j}-\sqrt{x_ix'_j}\right) u_{x_ix_j}(P')\right| + 2 \eta^{-1}\|\sqrt{x_ix_j}u_{x_i}\|_{C([0,T]\times\bar S_{n,m})}.
\end{aligned}
\end{equation}
Using inequality \eqref{eq:ADDS} and the fact that $|x'_j-x_j|\leq\eta$, by definitions of the points $P$ and $P'$, we have that
\begin{align*}
|\sqrt{x_ix_j}u_{x_ix_j}(P)| 
&\leq \eta^{\alpha/2}[\sqrt{x_ix_j}u_{x_ix_j}]_{C^{\alpha}_{WF}([0,T]\times\bar S_{n,m})}\\
&\quad + \sqrt{\frac{\eta}{x'_j}}\left|\sqrt{x_ix'_j} u_{x_ix_j}(P')\right|
       + 2 \eta^{-1}\|\sqrt{x_ix_j}u_{x_i}\|_{C([0,T]\times\bar S_{n,m})},
\end{align*}
which gives, by our assumption that $\eta<x'_j/2$,
\begin{equation}
\label{eq:InterpolationIneqS4_1}
\begin{aligned}
|\sqrt{x_ix_j}u_{x_ix_j}(P)| &\leq  \eta^{\alpha/2}[\sqrt{x_ix_j}u_{x_ix_j}]_{C^{\alpha}_{WF}([0,T]\times\bar S_{n,m})}\\
                             &\quad + \frac{1}{\sqrt{2}}\|\sqrt{x_ix_j} u_{x_ix_j}\|_{C([0,T]\times\bar S_{n,m})}
                             + 2\eta^{-1}\|\sqrt{x_ix_j}u_{x_i}\|_{C[0,T]\times\bar S_{n,m})},
\end{aligned}														
\end{equation}
which concludes this case.
\end{case}

\begin{case}[Points with $x_j$-coordinates large] We now assume that $\eta \geq x'_j/2$. Using the fact that $x'_j = x_j + \theta\eta$, for some $\theta\in [0,1]$, we have that $|x'_j-x_j| \leq x'_j$. From \cite[Lemma 5.2.5]{Epstein_Mazzeo_annmathstudies}, it follows that $\sqrt{x_ix_j} u_{x_ix_j}$ tends to $0$, as $x_j$ approaches $0$, and so, we obtain
\begin{align*}
\left|\left(\sqrt{x_ix_j}-\sqrt{x_ix'_j}\right) u_{x_ix_j}(P')\right| &\leq \left|\sqrt{x_ix'_j}u_{x_ix_j}(P')\right|
= \frac{\left|\sqrt{x_ix'_j}u_{x_ix_j}(P') - 0\right|}{\rho^{\alpha}(P',P^{''})} \rho^{\alpha}(P',P^{''}) \\
&\leq [\sqrt{x_ix_j}u_{x_ix_j}]_{C^{\alpha}_{WF}([0,T]\times\bar S_{n,m})} (2\eta)^{\alpha/2},
\end{align*}
where $P^{''}$ be the projection of $P'$ on the hyperplane $\{x_j=0\}$, which gives us the inequality $\rho(P',P^{''}) \leq \sqrt{x'_j}\leq\sqrt{2\eta}$, used in the second line of the preceding estimate. By inequality \eqref{eq:InterpolationIneq4_11}, we obtain that there is a positive constant, $C=C(\alpha)$, such that
\begin{equation}
\label{eq:InterpolationIneqS4_2}
\begin{aligned}
|\sqrt{x_ix_j}u_{x_ix_j}(P)| &\leq  C \eta^{\alpha/2}[\sqrt{x_ix_j}u_{x_ix_j}]_{C^{\alpha}_{WF}([0,T]\times\bar S_{n,m})}
                                    +2\eta^{-1}\|\sqrt{x_ix_j}u_{x_i}\|_{C([0,T]\times\bar S_{n,m})},
\end{aligned}
\end{equation}
which concludes this case.
\end{case}
Combining inequalities \eqref{eq:InterpolationIneqS4_1} and \eqref{eq:InterpolationIneqS4_2}, we obtain, for all $P\in [0,T]\times\bar S_{n,m}$, that
\begin{align*}
|\sqrt{x_ix_j}u_{x_ix_j}(P)| &\leq \frac{1}{\sqrt{2}}\|\sqrt{x_ix_j} u_{x_ix_j}\|_{C([0,T]\times\bar S_{n,m})}\\
&\quad +C \eta^{\alpha/2} [\sqrt{x_ix_j}u_{x_ix_j}]_{C^{\alpha}_{WF}([0,T]\times\bar S_{n,m})} 
+  2\eta^{-1}\|\sqrt{x_ix_j}u_{x_i}\|_{C([0,T]\times\bar S_{n,m})}.
\end{align*}
Rearranging terms yields
\begin{equation*}
\|\sqrt{x_ix_j}u_{x_ix_j}\|_{C([0,T]\times\bar S_{n,m})} \leq C \eta^{\alpha/2} [\sqrt{x_ix_j}u_{x_ix_j}]_{C^{\alpha}_{WF}([0,T]\times\bar S_{n,m})} 
+  C\eta^{-1}\|\sqrt{x_ix_j}u_{x_i}\|_{C([0,T]\times\bar S_{n,m})}.
\end{equation*}
Since $\eps\in (0,1)$, we may choose $\eta \in (0,1)$ in the preceding inequality such that $\eps= C\eta^{\alpha/2}$. To obtain inequality \eqref{eq:InterpolationIneqS4}, we then use \eqref{eq:InterpolationIneqS2}, together with the fact that the function $u$ has support in $[0,T]\times\bar M''_I$, and so, the the domain $M''_I$ is bounded in the $x_i$ and $x_j$-coordinates, for all $i,j\in I$. This concludes the proof of the interpolation inequality \eqref{eq:InterpolationIneqS4}, for all $j \in I$.
\end{step}

It remains to give the proofs of the interpolation inequalities \eqref{eq:InterpolationIneqS3}-\eqref{eq:InterpolationIneqS3_secund}. The proofs of inequalities \eqref{eq:InterpolationIneqS3_prim} and \eqref{eq:InterpolationIneqS3_secund} are similar to that of inequality \eqref{eq:InterpolationIneqS3}, but simpler, and so we only give the details of the proof of inequality \eqref{eq:InterpolationIneqS3}.

\begin{step}[Proof of inequality \eqref{eq:InterpolationIneqS3}]
\label{step:Proof_InterpolationIneqS3}
From inequality \eqref{eq:InterpolationIneqS2}, we see that it is sufficient to prove that estimate \eqref{eq:InterpolationIneqS3} holds for the H\"older seminorm $[x_i u_{x_i}]_{C^{\alpha}_{WF}([0,T]\times\bar S_{n,m})}$.  As in the proof of inequality \eqref{eq:InterpolationIneqS1}, it suffices to consider the differences $x_i^1u_{x_i}(P_1)-x_i^2u_{x_i}(P_2)$, where all except one of the coordinates of the points $P_1, P_2 \in [0,T]\times\bar S_{n,m}$ are identical. First, we consider the case when only the $x_i$-coordinates of the points $P_1$ and $P_2$ differ. 

\setcounter{case}{0}
\begin{case}[Points with $x_i$-coordinates close together] Assuming that $|x^1_i-x^2_i|\leq \eta$, and using the mean value theorem,  there is a point $P^*$ on the line segment connecting $P_1$ and $P_2$ such that
$$
x_i^1u_{x_i}(P_1)-x_i^2u_{x_i}(P_2)  = \left(x_i^* u_{x_i x_i}(P^*) + u_{x_i}(P^*)\right) (x_i^1-x_i^2).
$$
The argument used to prove Case \ref{case:InterpolationIneqS1_case1} of Step \ref{step:InterpolationIneqS1} applies immediately to this setting, with the role of function $u_{x_i}$ replaced by $x_iu_{x_ix_i}+u_{x_i}$, and we obtain that there is a positive constant, $C=C(\alpha)$, such that
\begin{equation}
\label{eq:InterpolationIneqS2_2}
\begin{aligned}
\frac{|x_i^1u_{x_i}(P_1)-x_i^2u_{x_i}(P_2)|}{\rho^{\alpha}(P_1,P_2)} &\leq C \eta^{1-\alpha} \|u\|_{C^{2+\alpha}_{WF}([0,T]\times\bar S_{n,m})},
\end{aligned}
\end{equation}
which concludes this case.
\end{case}

\begin{case}[Points with $x_i$-coordinates farther apart] Assuming that $|x^1_i-x^2_i|> \eta$, the argument of Case \ref{case:InterpolationIneqS1_case2} in Step \ref{step:InterpolationIneqS1} applies with $u$ replaced by $x_iu_{x_ix_i}$, and we obtain
\begin{equation*}
\begin{aligned}
\frac{|x_i^1u_{x_i}(P_1)-x_i^2u_{x_i}(P_2)|}{\rho^{\alpha}(P_1,P_2)}
&\leq C\eta^{-\alpha} \|x_i u_{x_i}\|_{C([0,T]\times\bar S_{n,m})}.
\end{aligned}
\end{equation*}
Since $\eps\in(0,1)$, we may choose $\eta$ such that $\eps=\eta^{\alpha+1}$ in inequality \eqref{eq:InterpolationIneqS2}, and we obtain
\begin{equation}
\label{eq:InterpolationIneqS2_xdfarapart}
\frac{|x_i^1u_{x_i}(P_1)-x_i^2u_{x_i}(P_2)|}{\rho^{\alpha}(P_1,P_2)}
\leq C \eta \|u\|_{C^{2+\alpha}_{WF}([0,T]\times\bar S_{n,m})} + C \eta^{-m_0(1+\alpha)-\alpha} \|u\|_{C([0,T]\times\bar S_{n,m})},
\end{equation}
which concludes this case.
\end{case}
Combining inequalities \eqref{eq:InterpolationIneqS2_2} and \eqref{eq:InterpolationIneqS2_xdfarapart} gives us that
\begin{equation}
\label{eq:InterpolationIneqS2_3}
\begin{aligned}
\frac{|x_i^1u_{x_i}(P_1)-x_i^2u_{x_i}(P_2)|}{\rho^{\alpha}(P_1,P_2)}
\leq C\eta^{1-\alpha} \|u\|_{C^{2+\alpha}_{WF}([0,T]\times\bar S_{n,m})} + C\eta^{-m_0(1+\alpha)-\alpha} \|u\|_{C([0,T]\times\bar S_{n,m})}.
\end{aligned}
\end{equation}

We now consider the case when only the $x_j$-coordinates, with $j\neq i$ and $j\in I$, differ. Let $x^k_j$ be the $x_j$-coordinates of the points $P_k$, for $k=1,2$, and assume that $x^1_j<x^2_j$. We have that
\begin{align*}
\sqrt{x_i}\left(u_{x_i}(P_2)-u_{x_i}(P_1)\right) &= \sqrt{x_i}\int_0^{x^2_j-x^1_j} u_{x_ix_j}(P_1+te_j)\, dt\\
&= \sqrt{x_i}\int_0^{x^2_j-x^1_j} \sqrt{x^1_j+t}\, u_{x_ix_j}(P_1+te_j)\frac{1}{\sqrt{x^1_j+t}}\, dt,
\end{align*}
which gives us that
\begin{align*}
\sqrt{x_i}\left|u_{x_j}(P_2)-u_{x_j}(P_1)\right| &\leq 2\|\sqrt{x_ix_j}u_{x_ix_j}\|_{C([0,T]\times\bar S_{n,m})} \left(\sqrt{x^2_j}-\sqrt{x^1_j}\right).
\end{align*}
Because the function $u$ has support in $[0,T]\times\bar M''_I$ and $j\in I$, we obtain from property \eqref{eq:Equivalent_intrinsic_metric} of the distance function $\rho$ that there is a positive constant, $C$, such that
\begin{align*}
\frac{\sqrt{x_i}\left|u_{x_j}(P_2)-u_{x_j}(P_1)\right|}{\rho^{\alpha}(P_1,P_2)} &\leq C\|\sqrt{x_ix_j}u_{x_ix_j}\|_{C([0,T]\times\bar S_{n,m})} \left(\sqrt{x^2_j}-\sqrt{x^1_j}\right)^{1-\alpha}.
\end{align*}
Using the fact that $u$ has support in $[0,T]\times\bar M''_I$, and that the set $M''_I$ is bounded in the $x_i$- and $x_j$-directions, from identity \eqref{eq:M_I_secund}, it follows that there is a positive constant, $C=C(\alpha)$, such that
\begin{align*}
\frac{\left|x_i(u_{x_j}(P_2)-u_{x_j}(P_1))\right|}{\rho^{\alpha}(P_1,P_2)} &\leq C\|\sqrt{x_ix_j}u_{x_ix_j}\|_{C([0,T]\times\bar S_{n,m})}.
\end{align*} 
From the interpolation inequality \eqref{eq:InterpolationIneqS4}, it follows that there are positive constants, $C=C(\alpha,m,n,T)$ and $m_0=m_0(\alpha,m,n)$, such that for all $\eps\in (0,1)$ we have that
\begin{align}
\label{eq:InterpolationIneqS2_10}
\frac{\left|x_i(u_{x_j}(P_2)-u_{x_j}(P_1))\right|}{\rho^{\alpha}(P_1,P_2)} &\leq \eps\|u\|_{C^{2+\alpha}_{WF}([0,T]\times\bar S_{n,m})} +C\eps^{-m_0}\|u\|_{C([0,T]\times\bar S_{n,m})}.
\end{align} 
A similar argument applied when only the $x_j$-coordinates, with $j\neq i$ and $j\in I^c$, or only the $y_l$-coordinates of the points $P_1$ and $P_2$, with $l=1,\ldots,m$, differ also yields the analogous inequality of \eqref{eq:InterpolationIneqS2_10}.

It remains to consider the case when only the $t$-coordinates of the points $P_1$ and $P_2$ differ. We denote $P_k=(t_k,z)$, for $k=1,2$, and we let $\gamma:=\sqrt{|t_1-t_2|}$.

\begin{case}[Points with $t$-coordinates close together] 
\label{case:Coordinate_t_1}
Assuming that $|t_1-t_2|<\eta$, we have
\begin{equation*}
\begin{aligned}
|u_{x_i}(P_1) - u_{x_i}(P_2)|
&\leq
\left|u_{x_i}(t_1,z) - \frac{1}{\gamma} \left( u(t_1,z+\gamma e_i) - u(t_1,z)\right)\right| \\
&\quad +\left|u_{x_i}(t_2,z) - \frac{1}{\gamma} \left( u(t_2,z+\gamma e_i) - u(t_2,z)\right)\right| \\
&\quad+ \frac{1}{\gamma} |u(t_1,z+\gamma e_i) - u(t_2,z+\gamma e_i)| +
   \frac{1}{\gamma} |u(t_1,z) - u(t_2,z)|.
\end{aligned}
\end{equation*}
and using the mean value theorem, there are points $t^*_k\in [0,T]$ and $P^*_k\in\bar S_{n,m}$, for $k=1,2$, such that
\begin{equation*}
\begin{aligned}
|u_{x_i}(P_1) - u_{x_i}(P_2)|
&= |u_{x_i}(t_1,z) - u_{x_i}(t_1,z+\theta_1\gamma e_i)| +
   |u_{x_i}(t_2,z) - u_{x_i}(t_2,z+\theta_2\gamma e_i)| \\
&\quad+  \frac{|t_1-t_2|}{\gamma} |u_t(t^*_1,z+\gamma e_i)| +
    \frac{|t_1-t_2|}{\gamma} |u_t(t^*_2,z)|\\
&\leq |u_{x_ix_i}(t_1,P_1^*)| \gamma + |u_{x_ix_i}(t_2,P_2^*) | \gamma\\
      &\quad + \frac{|t_1-t_2|}{\gamma} |u_t(t^*_1,z+\gamma e_i)| +
       \frac{|t_1-t_2|}{\gamma} |u_t(t^*_2,z)|.
\end{aligned}
\end{equation*}
Notice that $\rho(P_1,P_2) = \sqrt{|t_1-t_2|}=\gamma$ and so, by multiplying the preceding inequality by $x_i$, and using the fact that $u$ has support in $[0,T]\times\bar M''_I$, and that $M''_I$ is bounded in the $x_i$-coordinate, for all $i\in I$, we obtain
\begin{align*}
\frac{|x_iu_{x_i}(P_1) - x_iu_{x_i}(P_2)|}{\rho^{\alpha}(P_1,P_2)}
&\leq 2\|x_i u_{x_ix_i}\|_{C([0,T]\times\bar S_{n,m})} |t_1-t_2|^{\frac{1-\alpha}{2}}\\
&\quad+      2|t_1-t_2|^{1-\frac{1+\alpha}{2}}  \|x_iu_t\|_{C([0,T]\times\bar S_{n,m})},
\end{align*}
and thus, there is a positive constant, $C$, such that
\begin{equation}
\label{eq:InterpolationIneqS2_4}
\begin{aligned}
\frac{|x_iu_{x_i}(P_1) - x_iu_{x_i}(P_2)|}{\rho^{\alpha}(P_1,P_2)}
\leq C\eta^{\frac{1-\alpha}{2}}  \|u\|_{C^{2+\alpha}_{WF}([0,T]\times\bar S_{n,m})}.
\end{aligned}
\end{equation}
\end{case}

\begin{case}[Points with $t$-coordinates farther apart] 
\label{case:Coordinate_t_2}
Assuming that $|t_1-t_2|\geq\eta$, it immediately follows that
\begin{equation}
\label{eq:InterpolationIneqS2_5}
\begin{aligned}
\frac{|x_iu_{x_i}(P_1) - x_iu_{x_i}(P_2)|}{\rho^{\alpha}(P_1,P_2)}
&\leq& 2 \eta^{-\frac{\alpha}{2}} \|x_iu_{x_i}\|_{C([0,T]\times\bar S_{n,m})},
\end{aligned}
\end{equation}
which concludes this case.
\end{case}
By combining inequalities \eqref{eq:InterpolationIneqS2_4} and \eqref{eq:InterpolationIneqS2_5}, we obtain
\begin{equation}
\label{eq:InterpolationIneqS3_100}
\begin{aligned}
\frac{|x^1_iu_{x_i}(P_1) - x^2_iu_{x_i}(P_2)|}{\rho^{\alpha}(P_1,P_2)}
&\leq C\eta^{\frac{1-\alpha}{2}}  \|u\|_{C^{2+\alpha}_{WF}([0,T]\times\bar S_{n,m})} + 2 \eta^{-\frac{\alpha}{2}} \|u_{x_i}\|_{C([0,T]\times\bar S_{n,m})}.
\end{aligned}
\end{equation}
Combining inequalities \eqref{eq:InterpolationIneqS2_3}, \eqref{eq:InterpolationIneqS2_10} and \eqref{eq:InterpolationIneqS3_100} it follows that
\begin{equation*}
\begin{aligned}
\frac{|x_i^1u_{x_i}(P_1)-x_i^2u_{x_i}(P_2)|}{\rho^{\alpha}(P_1,P_2)}
&\leq C\left(\eta^{1-\alpha}+\eps\right) \|u\|_{C^{2+\alpha}_{WF}([0,T]\times\bar S_{n,m})} \\
&\quad+ 2 \eta^{-\frac{\alpha}{2}} \|u_{x_i}\|_{C([0,T]\times\bar S_{n,m})}+C\left(\eta^{-m_0(1+\alpha)-\alpha}+\eps^{-m_0}\right) \|u\|_{C([0,T]\times\bar S_{n,m})}.
\end{aligned}
\end{equation*}
Choosing $\eta=\eta(\eps)\in (0,1)$ small enough and using inequality \eqref{eq:InterpolationIneqS2}, we immediately obtain the interpolation inequality \eqref{eq:InterpolationIneqS3}. This concludes the proof of Step \ref{step:Proof_InterpolationIneqS3}.
\end{step}

This completes the proof of Proposition \ref{prop:InterpolationIneqS}.
\end{proof}

\section{Local a priori Schauder estimates}
\label{sec:Time_interior_estimate}
In this section we give the proofs of Theorems \ref{thm:Time_interior_estimate} and \ref{thm:Interior_estimate}. Our proof is based on a localization procedure of N.V. Krylov used in the proof of \cite[Theorem 8.11.1]{Krylov_LecturesHolder}, the interpolation inequalities for anisotropic H\"older spaces in Corollary \ref{cor:Higher_order_interpolation_inequalities}, and the global a priori Schauder estimates obtained in \cite[Theorem 10.0.2]{Epstein_Mazzeo_annmathstudies}. We begin with stating the properties of the coefficients of the differential operator $L$. 
\begin{assump}[Coefficients]
\label{assump:Coeff}
There is a nonnegative integer, $k$, and positive constants, $\delta$ and $K$, such that
\begin{enumerate}
\item[1.] The second-order coefficient functions satisfy the \emph{strict ellipticity condition}: for all sets of indices, $I\subseteq \{1,\ldots,n\}$, for all $z \in \bar M_I$, $\xi\in\RR^n$ and $\eta\in\RR^m$, we have 
\begin{equation}
\label{eq:Ellipticity}
\begin{aligned}
&\sum_{i\in I} a_{ii}(z)\xi_i^2
 +\sum_{i\in I^c} x_ia_{ii}(z)\xi_i^2
 + \sum_{i,j\in I} \tilde a_{ij}(z)\xi_i\xi_j
 +\sum_{i\in I} \sum_{j\in I^c} x_j(\tilde a_{ij}(z)+\tilde a_{ji}(z))\xi_i\xi_j\\
&+ \sum_{i,j\in I^c} x_ix_j\tilde a_{ij}(z)\xi_i\xi_j
 +\sum_{i\in I} \sum_{l=1}^m c_{il}(z)\xi_i\eta_l + \sum_{i\in I^c} \sum_{l=1}^m x_ic_{il}(z)\xi_i\eta_l
+\sum_{k,l=1}^m d_{kl}(z)\eta_k\eta_l\\
&\quad\geq \delta\left(|\xi|^2+|\eta|^2\right).
\end{aligned}
\end{equation}
\item[2.] The coefficient functions are \emph{H\"older continuous}: for all sets of indices, $I\subseteq \{1,\ldots,n\}$, we have 
\begin{equation}
\label{eq:Holder_cont}
\begin{aligned}
&\sum_{i\in I} \|a_{ii}\|_{C^{k,\alpha}_{WF}(\bar M_I)} 
+\sum_{i\in I^c} \|x_ia_{ii}\|_{C^{k,\alpha}_{WF}(\bar M_I)} 
+ \sum_{i,j\in I} \|\tilde a_{ij}\|_{C^{k,\alpha}_{WF}(\bar M_I)}\\
& +\sum_{i\in I} \sum_{j\in I^c} \left(\|x_j\tilde a_{ij}\|_{C^{k,\alpha}_{WF}(\bar M_I)}+\|x_j\tilde a_{ji}\|_{C^{k,\alpha}_{WF}(\bar M_I)}\right)
+ \sum_{i,j\in I^c} \|x_ix_j\tilde a_{ij}(z)\|_{C^{k,\alpha}_{WF}(\bar M_I)}\\
&+\sum_{i\in I} \sum_{l=1}^m \|c_{il}\|_{C^{k,\alpha}_{WF}(\bar M_I)} 
+ \sum_{i\in I^c} \sum_{l=1}^m \|x_ic_{il}\|_{C^{k,\alpha}_{WF}(\bar M_I)}
+\sum_{k,l=1}^m \|d_{kl}\|_{C^{k,\alpha}_{WF}(\bar M_I)}\\
& +\sum_{i=1}^n \|b_i\|_{C^{k,\alpha}_{WF}(\bar M_I)}
+\sum_{k=1}^m \|e_k\|_{C^{k,\alpha}_{WF}(\bar M_I)}\\
&\quad\leq K.
\end{aligned}
\end{equation}
\item[3.] The drift coefficient functions satisfy the \emph{nonnegativity condition}:
\begin{align}
\label{eq:b_coeff_bound}
b_{i}(z) \geq 0&\quad\hbox{on } \{z=(x,y)\in\partial S_{n,m}:\ x_i=0\},\quad \forall\, i=1,\ldots,n.
\end{align}
\end{enumerate}
\end{assump}

We can now give the
\begin{proof}[Proof of Theorem \ref{thm:Time_interior_estimate}]
Remark \ref{rmk:Reduction_to_one_set_of_indices} applies to the set $M'_I$ in place of $M''_I$, defined in \eqref{eq:M_I_prim} and \eqref{eq:M_I_secund}, respectively.  Thus, without loss of generality we may assume that $B_{2r}(z^0)\subset M'_{\{1,\ldots,n\}}$, when $r>0$ is chosen small enough. Let $L_0$ be defined similarly to the differential operator $L$, but with coefficients replaced by their values at $z^0$, that is,
\begin{equation*}
\begin{aligned}
L_0u&= \sum_{i=1}^n \left(x_ia_{ii}(z^0)u_{x_ix_i} + b_i(z^0)u_{x_i}\right) + \sum_{i,j=1}^n x_ix_j \tilde a_{ij}(z^0)u_{x_ix_j} \\
&\quad+\sum_{i=1}^n\sum_{l=1}^m x_i c_{il}(z^0)u_{x_iy_l}
 +\sum_{k,l=1}^m d_{kl}(z^0)u_{y_ky_l}+ \sum_{l=1}^m e_l(z^0)u_{y_l}.
\end{aligned}
\end{equation*}
Let $\varphi:\RR\rightarrow [0,1]$ be a smooth function such that $\varphi(t)=0$ for $t<0$, and $\varphi(t)=1$ for $t>1$. Let
$$
r_N=r\sum_{i=0}^{N} \frac{1}{2^i},\quad\forall\, N\in\NN,
$$
and consider the sequence of smooth cut-off functions, $\{\eta_N\}_{N\geq 1}\subset C^\infty_c(\bar S_{n,m})$, defined by
\[
\eta_N(z):=\varphi\left(\frac{r_{N+1}-|z-z^0|}{r_{N+1}-r_N}\right), \quad \forall\, z\in \bar S_{n,m},\quad\forall\, N\in\NN.
\]
We also let,
\[
T_N=\frac{T_0}{2} +\frac{T_0}{2}\left(2-\sum_{i=0}^{N} \frac{1}{2^i}\right),\quad\forall\, N\in\NN,
\]
and consider the sequence of smooth cut-off functions, $\{\psi_N\}_{N\geq 1}\subset C^\infty_c([0,\infty))$, defined by
\[
\psi_N(t):=\varphi\left(\frac{T_{N+1}-t}{T_{N+1}-T_N}\right), \quad \forall\, t\in [0,\infty),\quad\forall\, N\in\NN.
\]
We set $\varphi_N(t,z):=\psi_N(t)\eta_N(z)$, for all $(t,z)\in [0,T]\times \bar S_{n,m}$, and we let $Q_N:=[T_N,T]\times B_{r_N}(z^0)$. Then, we see that $0\leq \varphi_N\leq 1$, and $\varphi_N \equiv 1$ on $Q_N$, and $\varphi_N\equiv 0$ on $Q^c_{N+1}$, where $Q^c_{N+1}$ denotes the complement of $Q_{N+1}$ in $[0,T]\times\bar S_{n,m}$. 
By \cite[Lemma 10.1.3]{{Epstein_Mazzeo_annmathstudies}}, we can find a positive constant, $c=c(k,m,n,\varphi)$, such that for all $N\in\NN$, we have that
\begin{equation}
\label{eq:PropCutOffFunction}
\begin{aligned}
\|D^{\tau}_tD^{\zeta}_z\varphi_N\|_{C^{\alpha}_{WF}([0,T]\times\bar S_{n,m})}\leq c \rho^N\left(r^{-(k+3)}+T_0^{-(k+3)}\right),
\end{aligned}
\end{equation}
for all $\tau\in\NN$, $\zeta\in\NN^{n+m}$, such that $2\tau+|\zeta|=k+2$, where $\rho:=2^{k+3}>1$. We let
\begin{equation}
\label{eq:DefinitionAlpha}
\begin{aligned}
\alpha_N &:= \|u\varphi_N\|_{C^{k,2+\alpha}_{WF}([0,T]\times\bar S_{n,m})},\quad\forall\, N\in\NN.
\end{aligned}
\end{equation}
We may assume that $B_{2r}(z^0)$ is included in a compact manifold with corners, $P$ (see \cite[\S 2]{Epstein_Mazzeo_annmathstudies} for the definition of compact manifolds with corners). We extend the operator $L_0$ from $\bar B_{2r}(z^0)$ to $P$ such that it satisfies the hypotheses of \cite[Theorem 10.0.2]{Epstein_Mazzeo_annmathstudies}, and we obtain that there is a positive constant, $C=C(\alpha,\delta,k,K,m,n,T)$, such that
\begin{align}
\alpha_N &\leq  C\|(u\varphi_N)_t-L_0(u\varphi_N)\|_{C^{k,\alpha}_{WF}([0,T]\times\bar S_{n,m})}\notag\\
& \label{eq:ADDLocalAprioriBoundaryEst}
\leq C\|(u\varphi_N)_t-L(u\varphi_N)\|_{C^{k,\alpha}_{WF}([0,T]\times\bar S_{n,m})}
+C\|(L_0-L)(u\varphi_N)\|_{C^{k,\alpha}_{WF}([0,T]\times\bar S_{n,m})}.
\end{align}
We apply Lemma \ref{lem:Estimate_Lu} to the function $u\varphi_N$, the set $U=B_{r_N}(z_0)$, and with $L$ replaced by the operator $L_0-L$, to estimate the last term in the preceding inequality. Notice that because the coefficients of the operator $L$ are assumed to belong to $C^{\alpha}_{WF}([0,T]\times\bar S_{n,m})$, the constant $\Lambda$ in \eqref{eq:Constant_estimate_Lu} satisfies the bound $\Lambda\leq C r^{\alpha/2}$, where $C=C(K,m,n)$. It then follows that there are positive constants, $C=C(\alpha,k,K,m,n)$ and $m_k=m_k(\alpha,k,m,n)$, such that for all $\eps\in(0,1)$, we have
\begin{equation}
\label{eq:Estimate_second_term}
\begin{aligned}
\|(L_0-L)(u\varphi_N)\|_{C^{k,\alpha}_{WF}([0,T]\times\bar S_{n,m})} 
&\leq C(r^{\alpha/2}+\eps) \|u\varphi_N\|_{C^{k,2+\alpha}_{WF}([0,T]\times\bar S_{n,m})}\\
&\quad + C\eps^{-m_k}\|u\varphi_N\|_{C([0,T]\times\bar S_{n,m})}.
\end{aligned}
\end{equation}
We now estimate the first term on the right-hand side of inequality \eqref{eq:ADDLocalAprioriBoundaryEst}. We have that 
\begin{equation}
\label{eq:Operator_u_cutoff}
(u\varphi_N)_t-L(u\varphi_N) = \varphi_N(u_t- Lu) +u(\varphi_N)_t- [L,\varphi_N]u,
\end{equation}
where the last term is given by
\begin{equation}
\label{eq:ConstantCoeff2}
\begin{aligned}
\left[L,\varphi_N\right]u &= uL\varphi_N+2\sum_{i=1}^n  x_i a_{ii} u_{x_i} (\varphi_N)_{x_i} 
 +\sum_{i,j=1}^n x_ix_j\tilde a_{ij}\left(u_{x_i}(\varphi_N)_{x_j}+ u_{x_j}(\varphi_N)_{x_i} \right)\\
&\quad +\sum_{i=1}^n\sum_{l=1}^m x_i c_{il} \left(u_{x_i}(\varphi_N)_{y_l}+ u_{y_l}(\varphi_N)_{x_i} \right)
 +\sum_{l,k=1}^m d_{lk} \left(u_{y_l}(\varphi_N)_{y_k}+ u_{y_k}(\varphi_N)_{y_l} \right).
\end{aligned}
\end{equation}
Using \cite[Inequality (5.62)]{Epstein_Mazzeo_annmathstudies}, together with the fact that the support of the function $\varphi_N$ is included in the set $Q_{N+1}$, we have that
\begin{equation*}
\|\varphi_N (u_t- Lu)\|_{C^{k,\alpha}_{WF}([0,T]\times\bar S_{n,m})}
\leq 2\|\varphi_N\|_{C^{k,\alpha}_{WF}([0,T]\times\bar S_{n,m})} \|u_t- Lu\|_{C^{k,\alpha}_{WF}(\bar Q_{N+1})} ,
\end{equation*}
and estimate \eqref{eq:PropCutOffFunction} gives that there is a positive constant, $c=c(k,m,n,\varphi)$, such that
\begin{equation}
\label{eq:ConstantCoeff3}
\begin{aligned}
\|\varphi_N (u_t- Lu)\|_{C^{k,\alpha}_{WF}([0,T]\times\bar S_{n,m})}
\leq c \rho^N\left(r^{-(k+3)}+T_0^{-(k+3)}\right) \|u_t- Lu\|_{C^{k,\alpha}_{WF}([T_0/2,T]\times\bar B_{2r}(z^0))}.
\end{aligned}
\end{equation}
Because the coefficients of the operator $L$ satisfy inequality \eqref{eq:Holder_cont}, using estimate \eqref{eq:PropCutOffFunction}, we can find a positive constant, $C=C(\alpha,k,K,m,n)$, such that
\begin{equation*}
\begin{aligned}
&\|u(\varphi_N)_t -[L,\varphi_N]u\|_{C^{k,\alpha}_{WF}([0,T]\times\bar S_{n,m})}\\
&\qquad\leq C \rho^N\left(r^{-(k+3)}+T_0^{-(k+3)}\right) \left(\|u\varphi_{N+1}\|_{C^{k,\alpha}_{WF}([0,T]\times\bar S_{n,m})} \right.\\
&\qquad\quad\left.
+\sum_{i=1}^n\|x_i (u\varphi_{N+1})_{x_i}\|_{C^{k,\alpha}_{WF}([0,T]\times\bar S_{n,m})} 
+ \sum_{l=1}^m\|(u\varphi_{N+1})_{y_l}\|_{C^{k,\alpha}_{WF}([0,T]\times\bar S_{n,m})} \right).
\end{aligned}
\end{equation*}
The interpolation inequalities \eqref{eq:InterpolationIneqS1_higher_order}, \eqref{eq:InterpolationIneqS3_higher_order}, and \eqref{eq:InterpolationIneqS3_secund_higher_order}, together with the preceding inequality, equality \eqref{eq:Operator_u_cutoff} and estimate \eqref{eq:ConstantCoeff3} give us, for all $\eps\in(0,1)$,
\begin{equation}
\label{eq:ConstantCoeff5}
\begin{aligned}
&\|(u\varphi_N)_t-L(u\varphi_N)\|_{C^{k,\alpha}_{WF}([0,T]\times\bar S_{n,m})}\\
&\qquad\leq C \rho^N\left(r^{-(k+3)}+T_0^{-(k+3)}\right)
\left(\|u_t-Lu\|_{C^{k,\alpha}_{WF}([T_0/2,T]\times\bar B_{2r}(z^0))}\right.\\
&\qquad\quad\left.+\eps \|u\varphi_{N+1}\|_{C^{k,2+\alpha}_{WF}([0,T]\times\bar S_{n,m})} 
+ \eps^{-m_k} \|u\varphi_{N+1}\|_{C([0,T]\times\bar S_{n,m})} \right),
\end{aligned}
\end{equation}
where $C=C(\alpha,k,K,m,n,T)$ is a positive constant. Combining inequalities \eqref{eq:ADDLocalAprioriBoundaryEst}, \eqref{eq:Estimate_second_term} and \eqref{eq:ConstantCoeff5}, it follows that
\begin{equation}
\label{eq:Inequality_alpha_N}
\begin{aligned}
\alpha_N &\leq C \left(\eps\rho^N\left(r^{-(k+3)}+T_0^{-(k+3)}\right) +r^{\alpha/2}+\eps\right) \alpha_{N+1} \\
&\quad+ C \rho^N\left(r^{-(k+3)}+T_0^{-(k+3)}\right)\|u_t-Lu\|_{C^{k,\alpha}_{WF}([T_0/2,T]\times\bar B_{2r}(z^0))}\\
&\quad+ C \rho^N\left(r^{-(k+3)}+T_0^{-(k+3)}\right)\eps^{-m_k} \|u\varphi_{N+1}\|_{C([0,T]\times\bar S_{n,m})},
\end{aligned}
\end{equation}
where $C=C(\alpha,\delta,k,K,m,n,T)$ is a positive constant. Let $\gamma\in (0,1)$ be chosen such that
\begin{equation}
\label{eq:Choicedelta}
\rho^{1+m_k}\gamma \leq \frac{1}{2},
\end{equation}
Let $r_0=r_0(\alpha,k,m,n)$ be a positive constant such that 
$$
Cr_0^{\alpha/2} \leq \frac{\gamma}{3},
$$
and given any $r\in (0,r_0)$ and $T_0\in (0,T)$, choose $\eps=\eps(r,T_0)\in (0,1)$ such that
$$
C\eps\rho^N\left(r^{-(k+3)}+T_0^{-(k+3)}\right) = \frac{\gamma}{3},\quad\hbox{and}\quad C\eps\leq\frac{\gamma}{3}.
$$
Then we can rewrite inequality \eqref{eq:Inequality_alpha_N} in the form
\begin{align*}
\alpha_N &\leq \gamma \alpha_{N+1} 
+ C \rho^N\left(r^{-(k+3)}+T_0^{-(k+3)}\right) \|u_t-Lu\|_{C^{k,\alpha}_{WF}([T_0/2,T]\times\bar B_{2r}(z^0))}\\
&\quad+ (3C)^{1+m_k}\left(r^{-(k+3)}+T_0^{-(k+3)}\right)^{1+m_k}\gamma^{-m_k} \rho^{(1+m_k)N} \|u\varphi_{N+1}\|_{C([0,T]\times\bar S_{n,m})}.
\end{align*}
We multiply the preceding inequality by $\gamma^N$, and we let
\[
C_1:=\max\left\{C \left(r^{-(k+3)}+T_0^{-(k+3)}\right), (3C)^{1+m_k} \left(r^{-(k+3)}+T_0^{-(k+3)}\right)^{1+m_k}\gamma^{-m_k}\right\}.
\]
Then using inequality \eqref{eq:Choicedelta}, we obtain for all $r\in (0,r_0)$, 
\begin{equation*}
\begin{aligned}
\gamma^N \alpha_N &\leq \gamma^{N+1} \alpha_{N+1}
  +\frac{C_1}{2^N} \left(\|u_t-Lu\|_{C^{k,\alpha}_{WF}([T_0/2,T]\times\bar B_{2r}(z^0))}  + \|u\|_{C([T_0/2,T]\times\bar B_{2R}(z^0))} \right).
\end{aligned}
\end{equation*}
Summing the terms of the preceding inequality yields
\begin{equation*}
\begin{aligned}
\sum_{N=0}^{\infty} \gamma^N \alpha_N &\leq \sum_{N=0}^{\infty} \gamma^{N+1} \alpha_{N+1}
  + 2C_1\left(\|u_t-Lu\|_{C^{k,\alpha}_{WF}([T_0/2,T]\times\bar B_{2r}(z^0))}  + \|u\|_{C([T_0/2,T]\times\bar B_{2r}(z^0))} \right).
\end{aligned}
\end{equation*}
The sum $\sum_{N=0}^{\infty} \gamma^N \alpha_N$ is well-defined because we assumed that the function $u$ belongs to the space of functions $C^{k,2+\alpha}_{WF}([T_0/2,T]\times \bar B_{2r}(z^0))$, while $\gamma \in (0,1)$. By subtracting the term $\sum_{N=1}^{\infty} \gamma^N \alpha_N$ from both sides of the preceding inequality, we obtain the desired inequality \eqref{eq:Time_interior_estimate}.
\end{proof}

\begin{proof}[Proof of Theorem \ref{thm:Interior_estimate}]
We can use the same argument to prove Theorem \ref{thm:Interior_estimate} that we used in the proof of Theorem \ref{thm:Time_interior_estimate}, with the only modification that the function $\psi_N$ is no longer needed, and $T_N$ is chosen to be $0$, for all $N\in\NN$.
\end{proof}

\section{Existence and uniqueness of solutions}
\label{sec:Existence_solutions}

In this section, we give the proofs of Theorems \ref{thm:ExistenceUniqueness} and \ref{thm:ExistenceUniqueness_continuous}. The proof of Theorem \ref{thm:ExistenceUniqueness} relies on the a priori Schauder estimates established in \cite[Theorem 10.0.2]{Epstein_Mazzeo_annmathstudies}, and the supremum estimates derived from \cite[Proposition 3.3.1]{Epstein_Mazzeo_annmathstudies}. Finally, the existence result described in Theorem \ref{thm:ExistenceUniqueness_continuous}, for continuous initial data, is based on Theorem \ref{thm:ExistenceUniqueness} and a compactness argument which uses the a priori Schauder estimates in Theorem \ref{thm:Time_interior_estimate}.

We begin with

\begin{prop}[Comparison principle]
\label{prop:Comparison_principle}
Let $T>0$, and assume that the coefficients of the differential operator $L$ have the property that
$$
a_{ii}, \tilde a_{ij}, b_i, c_{il}, d_{lk}, e_l \in C([0,T]\times\bar S_{n,m}),
$$
for all $i,j=1,\ldots,n$ and all $l,k=1,\ldots,m$. Let $u$ be a function such that
$$
u \in C([0,T]\times\bar S_{n,m}) \cap C^1((0,T]\times\bar S_{n,m}) \cap C^2((0,T)\times S_{n,m}),
$$
and such that
$$
\lim_{x_i\downarrow 0} x_i u_{x_ix_i}(t,z)=0,\quad\forall\, t\in (0,T],\quad\forall\, i=1,\ldots,n.
$$
Assume that $u$ satisfies
\begin{align*}
u_t-Lu & \leq 0\quad\hbox{on } (0,T)\times\bar S_{n,m},\\
u(0,\cdot)&\leq 0\quad\hbox{on }  S_{n,m}.
\end{align*}
Then, we have that $u\leq 0$ on $[0,T]\times\bar S_{n,m}$.
\end{prop}

\begin{proof}
The proof is very similar to that of \cite[Proposition 3.1.1]{Epstein_Mazzeo_annmathstudies}, and so, we omit its detailed proof for brevity.
\end{proof}

We have the following consequence of Proposition \ref{prop:Comparison_principle}.
\begin{cor}[Maximum principle]
\label{cor:Maximum_principle}
Assume that the hypotheses of Proposition \ref{prop:Comparison_principle} hold, and that $u$ is a solution to the inhomogeneous initial-value problem \eqref{eq:Inhom_initial_value_problem}, where $f\in C(\bar S_{n,m})$ and $g\in C([0,T]\times\bar S_{n,m})$. Then
\begin{equation}
\label{eq:Maximum_principle}
\|u\|_{C([0,T]\times\bar S_{n,m})} \leq \|f\|_{C(\bar S_{n,m})} + T\|g\|_{C([0,T]\times\bar S_{n,m})}. 
\end{equation}
\end{cor}

\begin{proof}
We consider the auxiliary function,
$$
v(t,z):= \|f\|_{C([0,T]\times\bar S_{n,m})} + t\|g\|_{C([0,T]\times\bar S_{n,m})},\quad\forall\, (t,z)\in [0,T]\times\bar S_{n,m},
$$
and we apply Proposition \ref{prop:Comparison_principle} to $\pm u-v$. The supremum estimate \eqref{eq:Maximum_principle} follows immediately.
\end{proof}

We can now give the
\begin{proof}[Proof of Theorem \ref{thm:ExistenceUniqueness}]
Uniqueness of solutions is a straightforward consequence of Proposition \ref{prop:Comparison_principle}, and so, we only consider the question of existence of solutions. 
The proof of existence employs the method used in proving existence of solutions to parabolic partial differential equations outlined in \cite[Theorem II.1.1]{DaskalHamilton1998}. We let $\widehat C^{k,2+\alpha}_{WF}([0,T]\times\bar S_{n,m})$ denote the Banach space of functions $u\in C^{k,2+\alpha}_{WF}([0,T]\times\bar S_{n,m})$ such that $u(0,\cdot)\equiv 0$ on $\bar S_{n,m}$. Without loss of generality we may assume $f\equiv 0$ in the initial-value problem \eqref{eq:Inhom_initial_value_problem}, because the function $Lf \in C^{k,\alpha}_{WF}(\bar S_{n,m})$, when Assumption \ref{assump:Coeff} holds and $f\in C^{k,2+\alpha}_{WF}([0,T]\times\bar S_{n,m})$. We also have that
$$
\partial_t-L:\widehat C^{k,2+\alpha}_{WF}([0,T]\times\bar S_{n,m}) \rightarrow C^{k,\alpha}_{WF}([0,T]\times\bar S_{n,m})
$$
is a well-defined operator. Our goal is to show that $\partial_t-L$ is invertible and we accomplish this by constructing a bounded linear operator, $V:C^{k,\alpha}_{WF}([0,T]\times\bar S_{n,m}) \rightarrow \widehat C^{k,2+\alpha}_{WF}([0,T]\times\bar S_{n,m})$, such that
\begin{equation}
\label{eq:AlmostInverse}
\begin{aligned}
\left\|(\partial_t-L)V-I_{C^{k,\alpha}_{WF}([0,T]\times\bar S_{n,m})}\right\| <1.
\end{aligned}
\end{equation}
For this purpose, we fix $r>0$ and we choose a sequence of points, $\{z^N\}_{N\geq 1}$, such that the collection of balls $\{B_r(z^N)\}_{N\geq 1}$ covers the set $S_{n,m}\backslash M_{\emptyset}$. Without loss of generality, we may assume that there is a positive constant, $A=A(m,n)$, such that at most $A$ balls of the covering have non-empty intersection. Let $\{\varphi_N\}_{N\geq 0}\subset C^{\infty}_c(\bar S_{n,m})$ be a partition of unity subordinate to the open cover
\[
M_{\emptyset} \cup \bigcup_{N=1}^{\infty} B_r(z^N) = S_{n,m},
\]
such that
\[
\supp \varphi_0 \subset \{z \in S_{n,m}: \dist(z,\partial S_{n,m}) >r/2\},
\hbox{   and   }
\supp \varphi_N \subset \bar B_r(z^N), \quad \forall\, N\geq 1.
\]
We may choose the sequence of functions $\{\varphi_N\}_{N \geq 0}$ such that there is a positive constant, $c=c(k,m,n)$, such that
\begin{equation}
\label{eq:Eq0}
\begin{aligned}
\|\varphi_N\|_{C^{k,2+\alpha}_{WF}(\bar S_{n,m})} \leq c r^{-(k+3)}, \quad \forall\, r>0,\quad\forall\, N \geq 0.
\end{aligned}
\end{equation}
We choose a smooth function, $\psi_0\subset C^\infty(\bar S_{n,m})$, such that $0\leq \psi_0\leq 1$  and
$$
\psi_0(z) = \begin{cases} 0, &\hbox{on  } \{z \in S_{n,m}: \dist(z,\partial S_{n,m}) <r/8\}, \\ 
                          1, &\hbox{on  } \{z \in S_{n,m}: \dist(z,\partial S_{n,m}) >r/4\},
						\end{cases}
$$
and we choose a sequence of functions $\{\psi_N\}_{N\geq 1}$ such that $0\leq\psi_N\leq 1$ and $B_r(z^N)\subset \{\psi_N=1\}$. Thus, we have that
\begin{equation}
\label{eq:Eq1}
\psi_N \varphi_N = \varphi_N, \quad\forall\, N\geq 0,
\end{equation}
and without loss of generality we may assume that there is a positive constant, $c=c(k,m,n)$, with the property that
\begin{equation}
\label{eq:Eq0_psi}
\begin{aligned}
\|\psi_N\|_{C^{k,2+\alpha}_{WF}(\bar S_{n,m})} \leq c, \quad \forall\, r>0,\quad\forall\, N \geq 0.
\end{aligned}
\end{equation}
For $N=0$, let $L_0$ be a strictly elliptic operator on $\RR^{n+m}$ with $C^{k,\alpha}([0,T]\times R^{n+m})$-H\"older continuous coefficients, such that the operator $L_0$ agrees with $L$ on the support of the function $\psi_0$. This is possible due to our assumptions \eqref{eq:Ellipticity}, \eqref{eq:Holder_cont} and property \eqref{eq:Equivalent_intrinsic_metric} of the distance function $\rho$. We let
\begin{equation*}
\begin{aligned}
V_0: C^{k,\alpha}([0,T]\times\RR^{n+m})\rightarrow \widehat C^{k+2,\alpha}([0,T]\times\RR^{n+m}),
\end{aligned}
\end{equation*}
be the solution operator of $L_0$, that is, using \cite[Theorems 9.2.3 and 8.12.1]{Krylov_LecturesHolder}, we let $u:=V_0 g \in \widehat C^{k+2,\alpha}([0,T]\times\RR^{n+m})$ be the unique solution to the initial-value problem, $u_t-L_0u=g$ on $(0,T)\times\RR^{n+m}$, and $u(0,\cdot)=0$ on $\RR^{n+m}$, where $g\in C^{k,\alpha}([0,T]\times \RR^{n+m})$. For each $N\geq 1$, there is a set of indices, $I_N\subseteq \{1,2,\ldots,n\}$, such that $B_r(z^N)\subset M'_{I_N}$ and $\supp\psi_N\subset M''_{I_N}$. We then let $L_N$ be the degenerate-parabolic operator defined by
\begin{align*}
L_N u &:= \sum_{i\in I_N} x_i a_{ii}(z^N) u_{x_ix_i} + \sum_{i\notin I_N} x^N_i a_{ii}(z^N) u_{x_ix_i}+\sum_{i=1}^n b_i(z^N) u_{x_i}\\
&+\sum_{i,j\in I_N} x_ix_j \tilde a_{ij}(z^N) u_{x_ix_j} + \sum_{i\in I_N} \sum_{j\notin I_N} 2x_ix^N_j \tilde a_{ij}(z^N) u_{x_ix_j}
+\sum_{i,j\notin I_N} x_i^Nx_j^N\tilde a_{ij}(z^N) u_{x_ix_j}\\
&+\sum_{i\in I_N} \sum_{l=1}^m x_ic_{il}(z^N) u_{x_iy_l} +\sum_{i\notin I_N}\sum_{l=1}^m x_i^Nc_{il}(z^N)u_{x_iy_l}+\sum_{l,k=1}^m d_{lk}(z^N) u_{y_ly_k}+\sum_{k=1}^m e_l(z^N)u_{y_k}.
\end{align*}
Notice that the differential operator $L_N$ has the same structure as the operator $L$ defined in \eqref{eq:Generator}, with the observation that in this case, the number of `degenerate' coordinates, $n$, is replaced by $n_N:=|I_N|$, and the number of `non-degenerate' coordinates, $m$, is replaced by $m_N:=n+m-n_N$. By relabeling the coordinates, we may assume that the operator $L_N$ acts on functions defined on a compact manifold with corners, $P_N$ (see \cite[\S 2.1]{Epstein_Mazzeo_annmathstudies} for the definition of compact manifolds with corners). We choose the compact manifold $P_N$ such that 
$$
\supp \varphi_N \subset B_r(z^N) \subset \{\psi_N\equiv 1\} \subset \supp \psi_N \subset P_N,\quad\forall\, N\geq 1.
$$
We extend the coefficients of the operator $L_N$ from $\supp\psi_N$ to $P_N$ such that it satisfies the hypotheses of \cite[Theorem 10.0.2]{Epstein_Mazzeo_annmathstudies} to conclude that there is a solution operator,
\begin{equation*}
\begin{aligned}
V_N: C^{k,\alpha}_{WF}([0,T]\times P_N)\rightarrow \widehat C^{k,2+\alpha}_{WF}([0,T]\times P_N),
\end{aligned}
\end{equation*}
such that  $u:=V_N g \in \widehat C^{k,2+\alpha}_{WF}([0,T]\times P_N)$ is the unique solution to the initial-value problem, $u_t-L_Nu=g$ on $(0,T)\times P_N$, and $u(0,\cdot)\equiv 0$ on $P_N$, where $g\in C^{k,\alpha}_{WF}([0,T]\times P_N)$.

We can now define the operator
$$
V: C^{k,\alpha}_{WF}([0,T]\times\bar S_{n,m})\rightarrow \widehat C^{k,2+\alpha}_{WF}([0,T]\times\bar S_{n,m}),
$$
by setting
\begin{equation*}
\begin{aligned}
Vg := \sum_{N=0}^{\infty} \varphi_N V_N \psi_N g, \quad \forall\, g \in C^{k,\alpha}_{WF}([0,T]\times\bar S_{n,m}).
\end{aligned}
\end{equation*}
Our goal is to show that inequality \eqref{eq:AlmostInverse} holds, for small enough values of $r$ and $T$.  We have
\begin{equation*}
\begin{aligned}
(\partial_t-L)Vg-g &= \sum_{N=0}^{\infty} (\partial_t-L) \varphi_N V_N \left(\psi_N g\right) - g\\
&= \sum_{N=0}^{\infty}  \varphi_N (\partial_t-L) V_N \left(\psi_N g\right) - \sum_{N=0}^{\infty} [L,\varphi_N] V_N \left(\psi_N g\right) - g,
\end{aligned}
\end{equation*}
where the term $[L,\varphi_N]$ is given by \eqref{eq:ConstantCoeff2}. Denoting $u_N := V_N \left(\psi_N g\right)$, for all $N\geq 0$, we have
\begin{equation*}
\begin{aligned}
(\partial_t-L) V_N \left(\psi_N g\right) &= -(L-L_N) u_N + (\partial_t-L_N) V_N \left(\psi_N g\right)\\
&= -(L-L_N) u_N + \psi_N g,
\end{aligned}
\end{equation*}
since $(\partial_t-L_N)V_N = I$ on $\supp \psi_N$, for all $N\geq 0$. This implies, by identities \eqref{eq:Eq1} and the fact that $\{\varphi_N\}_{N\geq 0}$ is a partition of unity, that
\begin{equation}
\label{eq:Eq2}
\begin{aligned}
(\partial_t-L)Vg-g &= -\sum_{N=0}^{\infty}  \varphi_N (L-L_N) u_N - \sum_{N=0}^{\infty} [L,\varphi_N] u_N .
\end{aligned}
\end{equation}
We first estimate the terms in the preceding equality indexed by $N=0$. Because $L_0=L$ on the support of $\psi_0$, obviously we have that $\psi_0(L-L_0) u_0=0$. Next, using identity \eqref{eq:ConstantCoeff2}, there is a positive constant, $C=C(\alpha,k,K,m,n)$, such that
\begin{equation*}
\begin{aligned}
\left\|\left[L,\varphi_0\right]u_0\right\|_{C^{k,\alpha}([0,T]\times\bar S_{n,m})}
& \leq C \|\varphi_0\|_{C^{k+2,\alpha}([0,T]\times\bar S_{n,m})}\|u_0\|_{C^{k+1,\alpha}([0,T]\times\supp\varphi_0)} \\
& \leq C r^{-(k+3)}\|u_0\|_{C^{k+1,\alpha}([0,T]\times\supp\varphi_0)} \quad\hbox{(by \eqref{eq:Eq0}).}
\end{aligned}
\end{equation*}
Using the interpolation inequalities for standard H\"older spaces \cite[Theorem 8.8.1]{Krylov_LecturesHolder}, and the fact that the standard parabolic distance and the distance function $\rho$ are equivalent on $\supp \varphi_0$, by \eqref{eq:Equivalent_intrinsic_metric}, it follows that there is a positive constant, $m_k=m_k(\alpha,k,m,n)$, such that, for all $\eps\in (0,1)$, we have
\begin{equation}
\label{eq:Eq3}
\begin{aligned}
\left\|\left[L,\varphi_0\right]u_0\right\|_{C^{k,\alpha}_{WF}([0,T]\times\bar S_{n,m})}
& \leq C r^{-(k+3)} \left(\eps \|u_0\|_{C^{k,2+\alpha}_{WF}([0,T]\times\supp\varphi_0)}\right.\\
&\quad\left. + \eps^{-m_k} \|u_0\|_{C([0,T]\times\supp\varphi_0))} \right).
\end{aligned}
\end{equation}
By \cite[Theorem 8.12.1]{Krylov_LecturesHolder} and inequality \eqref{eq:Eq0_psi}, we have that there is a positive constant, $C=C(\alpha,k,K,m,n,T)$, such that
\begin{equation*}
\begin{aligned}
\|u_0\|_{C^{k,2+\alpha}_{WF}([0,T]\times \supp\varphi_0))} &\leq C r^{-(k+3)} \|g\|_{C^{k,\alpha}_{WF}([0,T]\times\bar S_{n,m})}.
\end{aligned}
\end{equation*}
From \cite[Corollary 8.1.5]{Krylov_LecturesHolder}, it follows that
$$
\|u_0\|_{C([0,T]\times\supp\varphi_0)} \leq T \|g\|_{C([0,T]\times\supp\psi_0)},
$$
and so, the preceding two inequalities together with \eqref{eq:Eq3}, give us that
\begin{equation}
\label{eq:Eq4}
\begin{aligned}
\left\|\left[L,\varphi_0\right]u_0\right\|_{C^{k,\alpha}_{WF}([0,T]\times \bar S_{n,m})}
& \leq Cr^{-(k+3)} \left(\eps \|g\|_{C^{k,\alpha}_{WF}([0,T]\times \bar S_{n,m})} + \eps^{-m_k} T\|g\|_{C([0,T]\times\bar S_{n,m})} \right).
\end{aligned}
\end{equation}
Next, we estimate the terms in identity \eqref{eq:Eq2} indexed by $N \geq 1$. From identity \eqref{eq:Eq1}, we have that
$\varphi_N (L-L_N) u_N = \varphi_N (L-L_N) (\psi_Nu_N)$. By Lemma \ref{lem:Estimate_varphi_Lu}, we have that there are positive constants, $C_r=C(\alpha,k,K,m,n,r,T)$ and $m_k=m_k(\alpha,k,m,n)$, such that for all $\eps\in (0,1)$, we have that
\begin{align*}
\|\varphi_N (L-L_N) u_N\|_{C^{k,\alpha}_{WF}([0,T]\times\bar S_{n,m})} 
&= \|\varphi_N (L-L_N) (\psi_Nu_N)\|_{C^{k,\alpha}_{WF}([0,T]\times\bar S_{n,m})} \\
&\leq \left(\Lambda \|\varphi_N\|_{C(\bar S_{n,m})}+C_r\eps\right)\|\psi_Nu_N\|_{C^{k,2+\alpha}_{WF}([0,T]\times\bar S_{n,m})}\\
&\quad + C_r\eps^{-m_k} \|\psi_Nu_N\|_{C([0,T]\times\bar S_{n,m})},
\end{align*}
Using the fact that $0\leq\psi_N\leq 1$ and $\supp\psi_N\subset P_N$, together with inequality \eqref{eq:Eq0_psi} and \cite[Inequality (5.62)]{Epstein_Mazzeo_annmathstudies}, we have that
\begin{align*}
\|\varphi_N (L-L_N) u_N\|_{C^{k,\alpha}_{WF}([0,T]\times\bar S_{n,m})} 
&\leq c\left(\Lambda \|\varphi_N\|_{C(\bar S_{n,m})}+C_r\eps\right)\|u_N\|_{C^{k,2+\alpha}_{WF}([0,T]\times P_N)}\\
&\quad + C_r\eps^{-m_k} \|u_N\|_{C([0,T]\times P_N)},
\end{align*}
Using the definition of the constant $\Lambda$ in \eqref{eq:Constant_estimate_Lu}, with $\bar U$ replaced by $\supp\varphi_N$, from our choice of the operator $L-L_N$, and property \eqref{eq:Holder_cont} of the coefficients of the operator $L$, we obtain that $\Lambda \leq C r^{\alpha/2}$. The preceding estimate becomes
\begin{align*}
\|\varphi_N (L-L_N) u_N\|_{C^{k,\alpha}_{WF}([0,T]\times\bar S_{n,m})} 
&\leq C\left(r^{\alpha/2}+C_r\eps \right)\|u_N\|_{C^{k,2+\alpha}_{WF}([0,T]\times P_N)}\\
&\quad + C_r\eps^{-m_k} \|u_N\|_{C([0,T]\times P_N)}.
\end{align*}
Applying \cite[Proposition 3.3.1]{Epstein_Mazzeo_annmathstudies} to the function $\pm u_N(t,z)-t\|g\|_{C([0,T]\times P_N)}$, we obtain that
$$
\|u_N\|_{C([0,T]\times P_N)} \leq T \|g\|_{C([0,T]\times P_N)},
$$
while the a priori Schauder estimates \cite[Theorem 10.0.2]{Epstein_Mazzeo_annmathstudies} show that there is a positive constant, $C=C(\alpha,\delta,k,K,m,n,T)$, such that
\begin{align*}
\|u_N\|_{C^{k,2+\alpha}_{WF}([0,T]\times\supp\varphi_N)} &\leq C \|\psi_Ng\|_{C^{k,\alpha}_{WF}([0,T]\times P_N)}\\
&\leq C \|g\|_{C^{k,\alpha}_{WF}([0,T]\times P_N)}\quad\hbox{(using inequality $\eqref{eq:Eq0_psi}$)}.
\end{align*}
From the preceding three inequalities, it follows that
\begin{equation}
\label{eq:Eq5}
\begin{aligned}
\|\varphi_N (L-L_N) u_N\|_{C^{k,\alpha}_{WF}([0,T]\times\bar S_{n,m})} 
&\leq C\left(r^{\alpha/2}+C_r\eps \right)\|g\|_{C^{k,\alpha}_{WF}([0,T]\times\bar S_{n,m})}\\
&\quad + C_r\eps^{-m_k} T\|g\|_{C([0,T]\times\bar S_{n,m})}.
\end{aligned}
\end{equation}
It remains to estimate the term $[L,\varphi_N]u_N$, for $N \geq 1$, by employing the same method that we used to estimate the term $[L,\varphi_0] u_0$. The only change is that we replace the standard interpolation inequalities \cite[Theorem 8.8.1]{Krylov_LecturesHolder} by Corollary \ref{cor:Higher_order_interpolation_inequalities}, the standard a priori Schauder estimates \cite[Theorems 9.2.3 and 8.12.1]{Krylov_LecturesHolder} by Theorem \ref{thm:Interior_estimate}, and the standard maximum principle \cite[Corollary 8.1.5]{Krylov_LecturesHolder} by \cite[Proposition 3.3.1]{Epstein_Mazzeo_annmathstudies}. We then obtain the analogue of inequality \eqref{eq:Eq4},
\begin{equation}
\label{eq:Eq9}
\begin{aligned}
&\left\|\left[L,\varphi_N\right]u_N\right\|_{C^{k,\alpha}_{WF}([0,T]\times\bar S_{n,m})}\\
&\qquad \leq C r^{-(k+3)} \left(\eps \|g\|_{C^{k,\alpha}_{WF}([0,T]\times\bar S_{n,m})} + \eps^{-m_k} T\|g\|_{C([0,T]\times\bar S_{n,m})} \right).
\end{aligned}
\end{equation}
Combining inequalities \eqref{eq:Eq4}, \eqref{eq:Eq5} and \eqref{eq:Eq9}, and using the fact that at most $A$ balls of the covering of $S_{n,m}$ have non-empty intersection, identity \eqref{eq:Eq2} yields
\begin{equation*}
\begin{aligned}
\|(\partial_t-L)Vg-g\|_{C^{k,\alpha}_{WF}([0,T]\times\bar S_{n,m})}
&\leq C\left(r^{\alpha/2}+\eps r^{-(k+3)}+\eps C_r\right)\|g\|_{C^{k,\alpha}_{WF}([0,T]\times\bar S_{n,m})}\\
&\quad + \left(Cr^{-(k+3)}\eps^{-m_k}+C_r\right) T\|g\|_{C([0,T]\times\supp\varphi_N)}.
\end{aligned}
\end{equation*}
By choosing the positive constants $r$, $\eps$ and $T$ small enough, we find a positive constant, $C_0<1$, such that
\begin{equation*}
\begin{aligned}
\|(\partial_t-L)Vg-g\|_{C^{k,\alpha}_{WF}([0,T]\times\bar S_{n,m})}
&\leq C_0\|g\|_{C^{k,\alpha}_{WF}([0,T]\times\bar S_{n,m})}, \quad \forall\, g \in C^{k,\alpha}_{WF}([0,T]\times\bar S_{n,m}),
\end{aligned}
\end{equation*}
which is equivalent to \eqref{eq:AlmostInverse}.

The preceding argument implies existence of solutions $u\in C^{k,2+\alpha}_{WF}([0,T]\times\bar S_{n,m})$, to the inhomogeneous initial-value problem \eqref{eq:Inhom_initial_value_problem}, with $f\equiv 0$ and $g \in C^{k,\alpha}_{WF}([0,T]\times\bar S_{n,m})$, up to a fixed time $T$. A standard bootstrapping argument allows us to obtain existence of solutions to problem \eqref{eq:Inhom_initial_value_problem} up to any time $T$. 

This completes the proof.
\end{proof}

Finally, we give the
\begin{proof}[Proof of Theorem \ref{thm:ExistenceUniqueness_continuous}]
Uniqueness of solutions is a straightforward consequence of Proposition \ref{prop:Comparison_principle}, and so, we only consider the question of existence of solutions. Let $\{f_N\}_{N\geq 1}\subset C^{\infty}(\bar S_{n,m})$ be a sequence of smooth functions such that
\begin{align}
\label{eq:Convergence_f_N}
\|f_N-f\|_{C([0,T]\times\bar S_{n,m})} \rightarrow 0,\quad\hbox{ as } N\rightarrow\infty,
\end{align}
Let $u_N$ be the unique solution to the inhomogeneous initial-value problem \eqref{eq:Inhom_initial_value_problem}, with $u_N(0,\cdot)= f_N$ on $\bar S_{n,m}$, given by Theorem \ref{thm:ExistenceUniqueness}. Then, it follows that $u_N\in C^{k,2+\alpha}_{WF}([0,T]\times\bar S_{n,m})$, for all $k\in\NN$ and all $\alpha\in (0,1)$. From Corollary \ref{cor:Maximum_principle} and property \eqref{eq:Convergence_f_N}, we obtain that
$$
\|u_N-u_M\|_{C([0,T]\times\bar S_{n,m})}  \leq \|f_N-f_M\|_{C([0,T]\times\bar S_{n,m})} \rightarrow 0,\quad\hbox{ as } N,M\rightarrow\infty,
$$
and so, the sequence $\{u_N\}_{N\geq 1}$ converges uniformly to a function $u\in C([0,T]\times\bar S_{n,m})$, and clearly we have that $u(0,\cdot)=f$ on $\bar S_{n,m}$. Moreover, applying Corollary \ref{cor:Maximum_principle} to each of the functions $u_N$, and using property \eqref{eq:Convergence_f_N}, we have that
\begin{equation}
\label{eq:Upper_bound_u}
\|u\|_{C([0,T]\times\bar S_{n,m})}  \leq C\|f\|_{C([0,T]\times\bar S_{n,m})} .
\end{equation}
Let $k\in\NN$, $\alpha\in (0,1)$, $T_0>0$ and let $r_0=r_0(\alpha,k,m,n)$ be the positive constant appearing in the conclusion of Theorem \ref{thm:Time_interior_estimate}. Covering $\bar S_{n,m}$ by a countable collection of balls $\{B_{r_0}(z^N)\}_{N\geq 1}$, we may apply estimate \eqref{eq:Time_interior_estimate} on each ball, to obtain that there is a positive constant, $C=C(\alpha,\delta,k,K,m,n,T_0,T)$, such that
$$
\|u_N\|_{C^{k,2+\alpha}_{WF}([T_0,T]\times \bar S_{n,m})} \leq 
C\left(\|\partial_t u_N-Lu_N\|_{C^{k,\alpha}_{WF}([0,T]\times\bar S_{n,m})}  + \|u_N\|_{C([0,T]\times\bar S_{n,m})} \right),
$$
and using inequality \eqref{eq:Upper_bound_u}, we have that
$$
\|u_N\|_{C^{k,2+\alpha}_{WF}([T_0,T]\times \bar S_{n,m})} \leq 
C\left(\|g\|_{C^{k,\alpha}_{WF}([0,T]\times\bar S_{n,m})}  + \|f\|_{C([0,T]\times\bar S_{n,m})} \right),\quad\forall\, N\geq 1.
$$
Thus, applying the Arzel\`a-Ascoli Theorem and \cite[Proposition 5.2.8]{Epstein_Mazzeo_annmathstudies}, we can find a subsequence of $\{u_N\}_{N\geq 1}$, which converges uniformly on compact subsets of $[T_0,T]\times\bar S_{n,m}$ in the H\"older space $C^{k,2+\alpha'}_{WF}([T_0,T]\times \bar S_{n,m})$, for all $\alpha'\in (0,\alpha)$, to a function that belongs to $C^{k,2+\alpha}_{WF}([T_0,T]\times \bar S_{n,m})$. Thus, the limit function, $u\in C([0,T]\times\bar S_{n,m})$, belongs to $C^{k,2+\alpha}_{WF}([T_0,T]\times \bar S_{n,m})$, satisfies the Schauder estimate \eqref{eq:Solution_estimate_continuous}, for all $k\in\NN$ and $\alpha\in (0,1)$, and solves the inhomogeneous initial-value problem \eqref{eq:Inhom_initial_value_problem}. This completes the proof.
\end{proof}

%
%

\bibliography{mfpde}

\def\cprime{$'$} \def\polhk#1{\setbox0=\hbox{#1}{\ooalign{\hidewidth
  \lower1.5ex\hbox{`}\hidewidth\crcr\unhbox0}}} \def\cprime{$'$}
  \def\cprime{$'$} \def\cprime{$'$}
  \def\lfhook#1{\setbox0=\hbox{#1}{\ooalign{\hidewidth
  \lower1.5ex\hbox{'}\hidewidth\crcr\unhbox0}}} \def\cprime{$'$}
  \def\cprime{$'$} \def\cprime{$'$} \def\cprime{$'$} \def\cprime{$'$}
\providecommand{\bysame}{\leavevmode\hbox to3em{\hrulefill}\thinspace}
\providecommand{\MR}{\relax\ifhmode\unskip\space\fi MR }
\providecommand{\MRhref}[2]{%
  \href{http://www.ams.org/mathscinet-getitem?mr=#1}{#2}
}
\providecommand{\href}[2]{#2}
\begin{thebibliography}{10}

\bibitem{Athreya_Barlow_Bass_Perkins_2002}
S.~R. Athreya, M.~T. Barlow, R.~F. Bass, and E.~A. Perkins, \emph{Degenerate
  stochastic differential equations and super-{M}arkov chains}, Probab. Theory
  Related Fields \textbf{123} (2002), 484--520.

\bibitem{Bass_Perkins_2003}
R.~F. Bass and E.~A. Perkins, \emph{Degenerate stochastic differential
  equations with {H}\"older continuous coefficients and super-{M}arkov chains},
  Trans. Amer. Math. Soc. \textbf{355} (2003), 373--405.

\bibitem{DaskalHamilton1998}
P.~Daskalopoulos and R.~Hamilton, \emph{{$C^\infty$}-regularity of the free
  boundary for the porous medium equation}, J. Amer. Math. Soc. \textbf{11}
  (1998), 899--965.

\bibitem{Epstein_Mazzeo_cont_est}
C.~L. Epstein and R.~Mazzeo, \emph{${C}^ 0$-estimates for degenerate diffusion
  operators arising in population biology}, pp. 65, preprint.

\bibitem{Epstein_Mazzeo_cont_est_diag}
\bysame, \emph{${C}^ 0$-estimates for diagonal degenerate diffusion operators
  arising in population biology}, pp. 19, preprint.

\bibitem{Epstein_Mazzeo_2010}
\bysame, \emph{Wright-{F}isher diffusion in one dimension}, SIAM J. Math. Anal.
  \textbf{42} (2010), 568--608.

\bibitem{Epstein_Mazzeo_annmathstudies}
\bysame, \emph{Degenerate diffusion operators arising in population biology},
  Annals of Mathematics Studies, Princeton University Press, Princeton, NJ,
  2013, arXiv:1110.0032.

\bibitem{Epstein_Pop_2013b}
C.~L. Epstein and C.~A. Pop, \emph{Harnack inequalities for degenerate
  diffusions}, pp. 55, preprint.

\bibitem{Ethier_Kurtz}
S.~N. Ethier and T.~G. Kurtz, \emph{Markov processes: Characterization and
  convergence}, Wiley, 1985.

\bibitem{Feehan_Pop_mimickingdegen_pde}
P.~M.~N. Feehan and C.~A. Pop, \emph{A {S}chauder approach to
  degenerate-parabolic partial differential equations with unbounded
  coefficients}, Journal of Differential Equations \textbf{254} (2013),
  4401--4445, arXiv:1112.4824.

\bibitem{Fisher_1922}
R.~A. Fisher, \emph{On the dominance ratio}, Proc. Roy. Soc. Edin. \textbf{42}
  (1922), 321--431.

\bibitem{Haldane_1932}
J.~B.~S. Haldane, \emph{The causes of evolution}, Harper and Brothers, New
  York, 1932.

\bibitem{KaratzasShreve1991}
I.~Karatzas and S.~E. Shreve, \emph{Brownian motion and stochastic calculus},
  second ed., Springer, New York, 1991.

\bibitem{KarlinTaylor2}
S.~Karlin and Taylor, \emph{A second course on stochastic processes}, Academic,
  New York, 1981.

\bibitem{Kimura_1957}
M.~Kimura, \emph{Some problems of stochastic processes in genetics}, Ann. Math.
  Statist. \textbf{28} (1957), 882--901.

\bibitem{Kimura_1964}
\bysame, \emph{Diffusion models in population genetics}, J. Appl. Probability
  \textbf{1} (1964), 177--232.

\bibitem{Krylov_LecturesHolder}
N.~V. Krylov, \emph{Lectures on elliptic and parabolic equations in {H}\"older
  spaces}, American Mathematical Society, Providence, RI, 1996.

\bibitem{Moser_1964}
J.~Moser, \emph{A {H}arnack inequality for parabolic differential equations},
  Comm. Pure Appl. Math. \textbf{17} (1964), 101--134.

\bibitem{Moser_1967}
\bysame, \emph{Correction to: ``{A} {H}arnack inequality for parabolic
  differential equations''}, Comm. Pure Appl. Math. \textbf{20} (1967),
  231--236.

\bibitem{Moser_1971}
\bysame, \emph{On a pointwise estimate for parabolic differential equations},
  Comm. Pure Appl. Math. \textbf{24} (1971), 727--740.

\bibitem{Pop_2013a}
C.~A. Pop, \emph{Existence, uniqueness and the strong {M}arkov property of
  solutions to {K}imura stochastic differential equations with singular drift},
  pp. 25, preprint.

\bibitem{Shimakura_1981}
N.~Shimakura, \emph{Formulas for diffusion approximations of some gene
  frequency models}, J. Math. Kyoto Univ. \textbf{21} (1981), no.~1, 19--45.

\bibitem{Wright_1931}
S.~Wright, \emph{Evolution in {M}endelian populations}, Genetics \textbf{16}
  (1931), 97--159.

\end{thebibliography}
\bibliographystyle{amsplain}

\end{document}